\documentclass[12pt]{amsart}
\usepackage{amssymb}
\usepackage{eufrak}
\usepackage{amsmath,tabu}
\usepackage{mathrsfs}
\usepackage[usenames, dvipsnames]{color}
\usepackage{amsthm}
\usepackage{ulem}

\definecolor{liens}{rgb}{1,0,0}
\usepackage[colorlinks=true, linkcolor=blue, 
hyperfootnotes=true,citecolor=blue,urlcolor=black]{hyperref}

\title{On the algebraic dependence of holonomic functions}
\author{Julien Roques} 
\address{Universit\'e de Lyon, Universit\'e Claude Bernard Lyon 1, CNRS UMR 5208, Institut Camille Jordan, F-69622 Villeurbanne, France}
\email{Julien.Roques@univ-lyon1.fr}
\author{Michael F. Singer}
\address{Department of Mathematics, North Carolina State University,
Box 8205, Raleigh, NC 27695-8205, USA}
\email{singer@math.ncsu.edu}

\date{\today}

\begin{document} 

\newtheorem{thm}{Theorem}
\newtheorem*{thmintro}{Theorem}
\newtheorem{lem}[thm]{Lemma}
\newtheorem{cor}[thm]{Corollary}
\newtheorem*{corintro}{Corollary}
\newtheorem{prop}[thm]{Proposition}
\newtheorem*{propintro}{Proposition}
\newtheorem{question}[thm]{Question}
\newtheorem*{questionintro}{Question}

\newtheorem{defin}[thm]{Definition}
\newtheorem{remark}[thm]{Remark}
\newtheorem{remarks}[thm]{Remarks}
\newtheorem{example}[thm]{Example}
\def\QED{\hbox{\hskip 1pt \vrule width4pt height 6pt depth 1.5pt \hskip 1pt}}
\newenvironment{prf}[1]{\trivlist
\item[\hskip \labelsep{\bf #1.\hspace*{.3em}}]}{~\hspace{\fill}~$\square$\endtrivlist}
\newenvironment{pproof}{
\begin{prf}{Proof}}{\end{prf}}
 \def\square{\QED}
\newcommand{\red}{\color{red}}
\newcommand{\blue} {\color{blue}}

\newcommand{\und}[1]{\boldsymbol{#1}}
\newcommand{\pFq}[5]{{}_{#1}F_{#2}\left[#3;#4;#5\right]}
\newcommand{\hypergeoequa}[4]{{}_{#1}\mathscr{H}_{#2}\left[#3;#4\right]}

\newenvironment{proofsln}{
\begin{prf}{Proof of Proposition~\ref{prop  case 2}}}{\end{prf}}

\newenvironment{proofitint}{
\begin{prf}{Proof of Theorem~\ref{prop:itint}}}{\end{prf}}

\def\ores{{\rm ores}}
\def\Cx{{\mathbb C}}
\def\CX{{\mathbb C}}
\def\RX{{\mathbb R}}
\def\QX{{\mathbb Q}}
\def\Qbar{\overline{\mathbb Q}}
\def\NX{{\mathbb N}}
\def\N{{\mathbb{Z}_{\geq 0}}}
\def\ZX{{\mathbb Z}}
\def\DX{{\mathbb D}}
\def\AX{{\mathbb A}}
\def\PX{{\mathbb P}}
\def\ord{\mbox{ord }}
\def\GL{{\rm GL}}
\def\SL{{\rm SL}}
\def\gl{{\rm gl}}
\def\sl{{\rm sl}}
\def\M{{\rm gl}}
\def\d{{
\partial}}
\def\fraka{{\mathfrak a}}
\def\frakk{{\mathfrak k}}
\def\frakg{{\mathfrak g}}
\def\frakK{{\mathfrak K}}
\def\frakE{{\mathfrak E}}
\def\frakF{{\mathfrak F}}
\def\frakU{{\mathfrak U}}
\def\frakp{{\mathfrak p}}
\def\ofrakF{{\overline{\mathfrak F}}}
\def\oofrakF{{\overline{\mathfrak F}}_0}
\def\ofrakG{{\overline{\mathfrak G}}}
\def\oofrakG{{\overline{\mathfrak G}}_0}
\def\frakG{{\mathfrak G}}
\def\td{{\tilde{D}}}
\def\ttd{{\tilde{\tilde{D}}}}
\def \ldf{{{\rm LDF}_0}}
\def\QED{\hbox{\hskip 1pt \vrule width4pt height 6pt depth 1.5pt \hskip 1pt}}
\def\calD{{\cal D}}
\def\calC{{\cal C}}
\def\calS{{\cal S}}
\def\calT{{\cal T}}
\def\calF{{\cal F}}
\def\calH{{\cal H}}
\def\calE{{\cal E}}
\def\calL{{\cal L}}
\def\calG{{\cal G}}
\def\calK{{\cal K}}
\def\calO{{\cal O}}
\def\calU{{\cal U}}
\def\calM{{\cal M}}
\def\calP{{\cal P}}
\def\calR{{\cal R}}
\def\calV{{\cal V}}
\def\calA{{\cal A}}
\def\calB{{\cal B}}
\def\U{{\mathcal U}}
\def\K{{\mathbf K}}
\def\cM{{\mathcal M}}
\def\cL{{\mathcal L}}
\def\rep{\mathrm{Repr}}
\def\calCF{{\calC_0(\calF)}}
\def\gal{{\rm DiffGal}}
\def\C{{\rm Const}}
\def\d{{
\partial}}
\def\dx{{
\partial_{x}}}
\def\dt{{
\partial_{t}}}
\def\Autd{{{\rm Aut}_\Delta}}
\def\Gal{{\rm Gal}}
\def\Hom{{\rm Hom}}
\def\PGal{{\rm Gal_\Delta}}
\def\Ga{{{\mathbb G}_a}}
\def\Gm{{{\mathbb G}_m}}
\def\ld{{\ell\d}}
\def\vr{{\Vec{r}}}
\def\va{{\Vec{\mathbf a}}}
\def\vb{{\Vec{\mathbf b}}}
\def\vc{{\Vec{\mathbf c}}}
\def\vf{{\Vec{\mathbf f}}}
\def\vy{{\Vec{\mathbf y}}}
\def\vta{{\Vec{\tau}}}
\def\vt{{\Vec{t}}}
\def\tY{{\tilde{Y}}}
\def\ty{{\tilde{y}}}
\def\sd{{\sigma\delta}}
\def\KPV{{K^{\rm PV}_A}}
\def\ord{{\rm ord}}
\def\seq{\mbox{\sc Seq}_C}
\def\abar{\overrightarrow{a}}
\def\bbar{\overrightarrow{b}}
\def\barg{\overline{g}}
\def\ybar{\overline{y}}
\def\Ker{{\rm{Ker}}}
\def\bK{{\mathbf{K}}}
\def\ba{{\mathbf{a}}}
\def\sd{{\sigma\partial}}
\def\tG{{\widetilde{G}}}
\def\hg{{\widehat{g}}}
\def\hh{{\widehat{h}}}
\def\des{{\rm{dis}}}
\def\dis{{\rm{dis}}}
\def\tr{\operatorname{tr}}
\def\CC{\mathbb{C}}
\def\RR{\mathbb{R}}
\def\QQ{\mathbb{Q}}
\def\ZZ{\mathbb{Z}}
\def\p3{\partial_3}

\bibliographystyle{amsalpha}
\sloppy
\begin{abstract}  We study the form of  possible algebraic relations between functions satisfying linear differential equations. In particular,
if $f$ and $g$ satisfy linear differential equations and are algebraically dependent, we give conditions on the differential Galois group associated to $f$ guaranteeing that $g$ is a polynomial in $f$. We apply this to  hypergeometric functions and iterated integrals.
\end{abstract}



\maketitle
\begin{quote}
\begin{center}
{\bf SUR LA D\'EPENDANCE ALG\'EBRIQUE DES FONCTIONS HOLONOMES}
\end{center}
\end{quote}

{\footnotesize
\begin{quote}
 \textsc{Résumé.} Nous nous intéressons aux relations algébriques vérifiées par des solutions d'équations différentielles linéaires. En particulier, si $f$ et $g$ satisfont des équations différentielles linéaires et sont algébriquement dépendantes, nous donnons des conditions sur le groupe de Galois différentiel de $f$ garantissant que $g$ est un polyn\^ome en $f$. Nous appliquons cela aux fonctions hypergéométriques et aux intégrales itérées. 
\end{quote}}

\setcounter{tocdepth}{1}
\tableofcontents

\section{Introduction}\label{sec : intro}

Let $K$ be a differential field of characteristic zero and let $k$ be an algebraically closed differential subfield of $K$ with the same field of constants $C$~:
 $$
 C=\{f \in K \ \vert \ f'=0 \}=\{f \in k \ \vert \ f'=0 \}.
 $$
 A typical example is the field of Puiseux series $K=\overline{C((x))}$ over the algebraically closed field $C$ endowed with the usual derivation $d/dx$ and $k=\overline{C(x)}$ the field of algebraic functions over $C(x)$.
 
 The present paper is concerned with the following general question. 
 
\begin{questionintro}\label{main question}
  We let $f$ and $g$ be elements of $K^{\times}$  satisfying some nontrivial homogeneous linear differential equations over $k$, say  
  $$
  L(f)=0 
  \text{ and } M(g)=0, 
  $$ 
  and we assume that $f$ and $g$ are algebraically dependent over $k$.  What can be said about $f$ and $g$? about the algebraic relations  between $f$ and $g$ over $k$? 
\end{questionintro}

 Differential Galois theory  has been an efffective tool in understanding the possible algebraic  relations {among} solutions of linear differential equations (see, for example, \cite{Kolchin68, HarrisSibuya85, HarrisSIbuya86, Sperber, Singer86,  BBH, roquesOGHEMM}). A celebrated illustration of this is the following  result of Kolchin answering Question \ref{main question} when $L$ and $M$ are first order equations~: 
 
\begin{thm}[Kolchin \cite{Kolchin68}]
If $L$ and $M$ have order $1$, then $f$ and $g$ are algebraically dependent over $k$ if and only if there exists $(m,n) \in \mathbb{Z}^{2} \setminus \{(0,0)\}$ such that 
 $$
 f^{m}g^{n} \in k.
 $$ 
\end{thm}


The present work started with the following question : what can be said if we assume that $L$ has order $1$, but do not make any assumption on the order of $M$ ? Our answer reads as follows~: 

\begin{thm}[Theorem~\ref{prop:torus} in Section~\ref{sec: order one eq}]\label{thm order 1 intro}
 If $L$ has order $1$ and if $f \not \in k$, then $f$ and $g$ are algebraically dependent over $k$ if and only if there exists $\theta$ in an algebraic extension of $K$ and a positive integer $n$ such that 
$$
f=\theta^{n}
\text{ and }
g \in k[\theta,\theta^{-1}]. 
$$
\end{thm}
The rest of the paper is devoted to the case when the operator $L$ has order $\geq 2$. 
For any $f \in K^{\times}$ as above, we let $\mathcal{A}(f)$ be the $k$-algebra made of the elements $g$ of $K$ holonomic and algebraically dependent  on $f$ over $k$. In general, the obvious inclusion 
$$
k[f] \subset \mathcal{A}(f)
$$ 
is not an equality. The principal aim of Sections~\ref{sec:simpcon}, \ref{sec:specialcase} and \ref{sec:genred} is to give criteria on the differential Galois group of $L$ over $k$ ensuring that $\mathcal{A}(f)=k[f]$. In this introduction, we only state the following concrete consequence of these criteria and refer to Sections~\ref{sec:simpcon}, \ref{sec:specialcase} and \ref{sec:genred} for more general statements. 

\begin{thm}[Consequence of Proposition~\ref{prop  case 2} in Section~\ref{sec:simpcon} and Proposition~\ref{prop: case GL} in Section~\ref{sec:specialcase}]\label{thm SL or GL intro}
 Assume that $L$ has order $n \geq 2$ and that its differential Galois group $G_{L}$ over $k$ is either $\SL_{n}(C)$ or $\GL_{n}(C)$. Then, we have 
 $$\mathcal{A}(f) = k[f].$$
 \end{thm}

\begin{remark}
1.~When $n=2$, the fact that the differential Galois group over $k$ of $L$ is either $\SL_{2}(C)$ or $\GL_{2}(C)$  is equivalent to the irreducibility of $L$ over $k$ \cite{AlgSubgroups}. We thus have the following statement :
 If $L$ has order $2$ and is irreducible over $k$ then 
  $\mathcal{A}(f) = k[f].$\\[0.05in]
2. Another simple consequence of Theorem \ref{thm SL or GL intro} is: Assume that $L$ has order $n \geq 2$ and that its differential Galois group $G_{L}$ over $k$ is either $\SL_{n}(C)$ or $\GL_{n}(C)$. If $f \not \in k$, then, for $m \geq 2$, $f^{1/m}$ does not satisfy any nontrivial linear differential equation over $k$. This follows easily form the fact that we cannot have $f^{1/m} \in k[f]$ unless $f$ is algebraic over $k$ and, hence, belongs to $k$. For related results, see \cite{HarrisSIbuya86, Sperber, Singer86}.
\end{remark}

Theorem \ref{thm SL or GL intro} applies, for example, to many generalized hypergeometric series
\begin{equation*}
\pFq{p}{q}{\und \alpha}{\und \beta}{x}=\sum_{k=0}^{+\infty} \frac{(\alpha_{1})_k \cdots (\alpha_{p})_k}{(\beta_{1})_k \cdots (\beta_{q})_k} \frac{x^k}{k!} \in \CC((x))
\end{equation*}
where $\und \alpha =(\alpha_{1},\ldots,\alpha_{p}) \in \CC^{p}$ and $\und \beta =(\beta_{1},\ldots,\beta_{q}) \in (\CC\setminus \mathbb{Z}_{\leq 0})^{q}$ for some $p,q \in \ZZ_{\geq 0}$ and  
where the Pochhammer symbols $(t)_{k}$ are defined by $(t)_0=1$ and, for $k \in \mathbb{Z}_{\geq 1}$, $(t)_k=t(t+1)\cdots (t+k-1)$. Indeed, this series satisfies the generalized hypergeometric differential equation 
$$
\hypergeoequa{p}{q}{\und \alpha}{\und \beta}(\pFq{p}{q}{\und \alpha}{\und \beta}{x}) =0 
$$
where 
\begin{equation*} 
\hypergeoequa{p}{q}{\und \alpha}{\und \beta} = \delta \prod_{k=1}^{q} (\delta + \beta_k-1) - x \prod_{k=1}^{p} (\delta + \alpha_k)
\end{equation*}
with $\delta=x\frac{d}{dx}$ and the work of Beukers, Brownawell and Heckman \cite{BBH}, Beukers and Heckman \cite{beukersheckman},  Duval and Mitschi \cite{duvalmitschi}, Katz \cite{katzcalculation,expsum} and Mitschi \cite{mitschi} give explicit conditions (holding for generic parameters $\und \alpha,\und \beta$) ensuring that the differential Galois group over $\overline{\CC(x)}$ of $\hypergeoequa{p}{q}{\und \alpha}{\und \beta}$ is either $\SL_{n}(\CC)$ or $\GL_{n}(\CC)$ (in these references, the differential Galois groups are computed over $\CC(x)$, not $\overline{\CC(x)}$; {this is not a problem since {if the differential Galois group over ${\CC(x)}$ is connected then it does not change when one replaces the base field with $\overline{\CC(x)}$ \cite[Proposition 5.28]{Ma}.}} 
For instance, it follows from \cite[Theorem 3.6]{expsum} that, for any $\beta \in \QQ \setminus \ZZ_{\leq 0}$, the differential Galois group over $\overline{\CC(x)}$ of $\hypergeoequa{0}{1}{-}{\beta}$ is $\SL_{2}(\CC)$; so, for $K=\overline{\CC((x))}$ endowed with the derivation $d/dx$ and $k=\overline{\CC(x)}$, Theorem \ref{thm SL or GL intro} ensures that~:

\begin{cor}\label{cor for 0F1 intro}
For any $\beta \in \QQ \setminus \ZZ_{\leq 0}$, we have 
$$
\mathcal{A}(\pFq{0}{1}{-}{\beta}{x}) = \overline{\CC(x)}[\pFq{0}{1}{-}{\beta}{x}].
$$
\end{cor}

Moreover, we show in Section~\ref{sec:hyper} that our methods (combined with the local formal theory of differential equations) can be used in order to examine in details the algebraic dependence relations  between hypergeometric series. For instance, Corollary~\ref{cor for 0F1 intro} ensures that if some generalized hypergeometric series $\pFq{p}{q}{\und \gamma}{\und \delta}{x}$ is algebraically dependent on $\pFq{0}{1}{-}{\beta}{x}$, then it is a polynomial in $\pFq{0}{1}{-}{\beta}{x}$; in Section~\ref{sec:hyper}, we establish a much more precise statement~: 

\begin{thm}\label{thm intro : application to alg rel hypergeo}
Consider $\beta \in \QQ \setminus \ZZ_{\leq 0}$. Consider  $\und \gamma \in (\CC \setminus \ZZ_{\leq 0})^{p}$ and $\und \delta \in (\CC \setminus \ZZ_{\leq 0})^{q}$. If $\pFq{0}{1}{-}{\beta}{x}$ and $\pFq{p}{q}{\und \gamma}{\und \delta}{x}$ are algebraically dependent over $\overline{\CC(x)}$, then 
\begin{itemize}
 \item either 
 $$
q=p+1 \text{ and }   
\pFq{p}{q}{\und \gamma}{\und \delta}{x}  \in \CC(x) \pFq{0}{1}{-}{\beta}{x} 
+
\CC(x); 
$$
\item or 
$$q  = p -1 \text{ and } \pFq{p}{q}{\und \gamma}{\und \delta}{x} \in \overline{\CC(x)}.
$$ 
\end{itemize}
\end{thm}

\begin{remark}
 Whether or not the second condition holds true  can be decided by using  Beukers and Heckman's \cite{beukersheckman} provided that $\gamma_{i}-\delta_{j} \not \in \ZZ$ for all $i \in \{1,\ldots,p\}$ and $j \in \{1,\ldots,q\}$. 
\end{remark}

As already mentioned, Theorem \ref{thm SL or GL intro} is a consequence of results contained in Sections~\ref{sec:simpcon}, \ref{sec:specialcase} and \ref{sec:genred}, where the hypothesis that the differential Galois group $G_{L}$ of $L$ over $k$ is either $\SL_{n}(C)$ or $\GL_{n}(C)$ is replaced by the hypothesis that $G_{L}$ is either simply connected or reductive. These generalizations (more precisely, the generalization to the simply connected case) allow us to describe $\mathcal{A}(f)$ when $f$ is an iterated integral ({\it i.e.}, when $f^{(n)} \in k$ for some positive integer $n$, see \cite{Ravi}) and to study the algebraic relations among iterated integrals. Indeed, in Section \ref{sec:multint} we prove the following results. 

\begin{thm}[Theorem \ref{prop:itint} in Section \ref{sec:multint}]\label{sec:multint intro}  Let $f \in K^\times \setminus k$ be an iterated integral over $k$ and let $g \in K$ satisfy a nonzero homogeneous linear differential equation over $k$.  Assume that $k$ contains an element $x$ with $x'=1$.  Then, we have $$\mathcal{A}(f) = k[f].$$ \end{thm} 

\begin{prop}[Proposition \ref{prop:liouv} in Section \ref{sec:multint}]\label{prop:liouv intro} Assume that there is an element $x\in k$ such that $x'=1$. Let $f_1, \ldots , f_n$ be iterated integrals over $k$. If $f_1, \ldots , f_n$  are algebraically dependent over $k$ then there exist $u_1, \ldots , u_n \in k$, not all zero, such that 
$$
u_1f_1+ \cdots + u_n f_n \in k.
$$ 
\end{prop}%

\begin{remark}
Our proof of Proposition \ref{prop:liouv intro} does not use the fact that $k$ is algebraically closed, it only requires that $C$  be algebraically closed. 
\end{remark}

This paper is organized as follows.  Section~\ref{sec: order one eq} is devoted to  the case where $L$ has order $1$ and contains a proof of Theorem \ref{thm order 1 intro} above. In Section~\ref{sec:simpcon} we give a result  when the differential Galois group $G_L$ of $L$ over $k$ is simply connected, apply  this to the case when $G_L = \SL_n(C)$ (proving the $\SL_{n}(C)$-case of Theorem~\ref{thm SL or GL intro} above). 
In Section~\ref{sec:specialcase} we consider the case when $G_L$ is $\GL_n(C)$ (proving the $\GL_{n}(C)$-case of Theorem~\ref{thm SL or GL intro} above) and in Section~\ref{sec:genred} the case when $G_L$ is a general reductive group. {We have included a full treatment of the $\GL_{n}(C)$-case before considering the general reductive case because the proof is less technical in the $\GL_{n}(C)$-case but already contains most of the ideas used in the general reductive case.}  In Section~\ref{sec:hyper} we apply the results of the previous sections to generalized hypergeometric series and prove Theorem \ref{thm intro : application to alg rel hypergeo} above. In Section~\ref{sec:multint} we study the algebraic relations among iterated integrals (proving Proposition \ref{sec:multint intro} and Proposition \ref{prop:liouv intro} above). The Appendix gives an informal introduction to the tannakian approach to the differential Galois theory - an approach that appears in several of our proofs.

\section{Notations and {conventions}}\label{sec:notations}

In the whole paper, except where noted to the contrary ({\it e.g.}, in Section~\ref{sec:multint}):
\begin{itemize}
 \item $K$ is a differential field of characteristic zero;
 \item $k$ is an algebraically closed differential subfield of $K$; 
 \item we assume that $K$ and $k$ have the same field of constants $C$~:
 $$
 C=\{f \in K \ \vert \ f'=0 \}=\{f \in k \ \vert \ f'=0 \}.
 $$
\end{itemize}
 Note that $C$ has characteristic zero and is algebraically closed because it is the field of constants of $k$ which itself has characteristic zero and is algebraically closed (see \cite[Exercise 1.5.2]{VdPS}). Also  note  that the assumption that $k$ is algebraically closed implies that the differential Galois group of any Picard-Vessiot extension is connected \cite[Proposition 1.34]{VdPS}.
 
{We will assume that the reader is familiar with some basic notions concerning linear algebraic groups (a general reference is \cite{Humphreys}) and  the basic differential Galois theory (a general reference is \cite{Ma} or \cite{VdPS}). In our proofs we will also sometimes use the tannakian correspondence. Briefly this states that for a differential equation $Y'=AY$, there is a correspondence between representations of the associated differential Galois group and equations gotten from the original equations via applying the constructions of linear algebra (direct sums, tensor products, duals, submodules and quotients) to the differential module associated with $Y'=AY$.  In the Appendix we give a fuller but still informal description of this. Formally this is given as an equivalence of categories; for precise details we refer to \cite{delignemilne,KatzAlgSolDiffEqua,katzcalculation,VdPS}).

Except with explicit mention to the contrary, by ``algebraic group'', we will mean ``linear algebraic group over $C$''. }

\section{The first order case}\label{sec: order one eq}

{We start this Section with a preliminary result concerning  algebraic groups. This  result  will be used several times in this paper. }

\begin{lem}\label{lem:torus isogeny}
 Let $\varphi : G_{1} \rightarrow G_{2}$ be an isogeny between two algebraic groups $G_{1}$ and $G_{2}$. If $G_{1}$ is connected and if $G_{2}$ is a torus, then  $G_{1}$ is a torus as well and $G_{1}$ and $G_{2}$ have the same dimension. 
\end{lem}

\begin{pproof}
 To show that the connected group $G_{1}$ is a torus it is enough to show that it consists of semisimple elements (c.f. \cite[Exercise 2 in Section 21.4]{Humphreys}). It is equivalent to showing that $e_{G_{1}}$ is the unique unipotent element of $G_{1}$ (because any linear algebraic group over $C$ contains the semisimple and the unipotent parts of the Jordan decompositions of its elements).  Since the image by 
$
\varphi
$ 
of a unipotent element of $G_{1}$ is unipotent and since the only unipotent element of $G_{2}$ is $e_{G_{2}}$, the unipotent elements of $G_{1}$ are all in $\ker \varphi$. Since $\ker \varphi$ is finite and unipotent elements have infinite order or are the iedentity, $G_1$ has a unique unipotent element, namely $e_{G_{1}}$. 

{It remains to prove that $G_{1}$ and $G_{2}$ have the same dimension. This {follows from the fact that $\dim G_1 = \dim \ker \phi + \dim {\rm im} \,  \phi$, \cite[CH.~7.3, Proposition B]{Humphreys}}}
\end{pproof}

\begin{thm}\label{prop:torus}
Let {$f \in K^{\times}\setminus k$} be such that 
$$
f'=af
$$ 
for some $a \in k$.  Let $g \in K$ satisfy a nonzero homogeneous linear differential equation over $k$. We then have that $f$ and $g$ are algebraically dependent over $k$ if and only if there exists $\theta$ in an algebraic extension of $K$ and an  integer $n$ such that 
\begin{align}\label{eq:algdep1}
f=\theta^{n}
\text{ and }
g \in k[\theta,\theta^{-1}]. 
\end{align}
\end{thm}

\begin{pproof} If $f$ and $g$ satisfy \eqref{eq:algdep1}, they are clearly algebraically dependent over $k$. 

Let us prove the converse implication. 
Let $L$ be a nonzero linear differential operator with coefficients in $k$ annihilating $g$ and of minimal order $r$ for this property. Let $Y'=A_{L}Y$ be the corresponding linear differential system. 

We let $E$ be a Picard-Vessiot extension over $k$ for the differential system  
\begin{equation}\label{diff syst direct sum}
 Y'=
\begin{pmatrix}
 a&0\\
 0&A_{L}
\end{pmatrix}
Y
\end{equation}
containing $f$ and $g$. Let $G_E$ be the corresponding differential Galois group. The differential system \eqref{diff syst direct sum} has a fundamental solution of the form  
$$
\mathfrak{Y} 
=
\begin{pmatrix}
 f&0&\cdots &0\\
 0&g_{1}&\cdots &g_{r}\\
 0&g_{1}'&\cdots &g_{r}'\\
  0&\vdots&\ddots &\vdots\\
   0&g_{1}^{(r-1)}&\cdots &g_{r}^{(r-1)}
\end{pmatrix} 
\in GL_{r+1}(E)
$$
with $g_{1}=g$ (thus, $E$ is the field extension of $k$ generated by the entries of  $\mathfrak{Y}$). 
The field extension $E_{1}$ (resp. $E_{2}$) of $k$ generated by $f$ (resp. by the $g_{i}^{(j)}$) is a  Picard-Vessiot extension over $k$ for $y'=ay$ (resp. $Y'=A_{L}Y$); we let $G_{1}$ (resp. $G_{2}$) be the corresponding differential Galois group over $k$. 

By minimality of $L$, we have 
$$
\operatorname{span}_{C} G g
=
\operatorname{span}_{C} G_{2} g
=
\operatorname{span}_{C}\{g_{1},\ldots,g_{r}\}
$$ 
(note that $Gg=G_{2}g$ because the restriction morphisms $G \rightarrow G_{2}$ is surjective). 
Moreover, the $C$-line generated by $f$ is left invariant by $G$. These facts, together with the fact that $g$ is algebraic over $k(f)$, imply that the $g^{(j)}_{i}$ are algebraic over $k(f)$ as well, {\it i.e.}, $E$ is a finite extension of $E_{1}$. 

{The differential Galois theory implies that} the restriction morphism  
$
G \rightarrow G_{1}
$
is surjective and hence induces a short exact sequence 
$$
0 \rightarrow H \rightarrow G \rightarrow G_{1} \rightarrow 0. 
$$
The kernel $H$ (which is the group of differential field automorphisms of $E$ over $E_{1}$) is finite because $E$ is a finite extension of $E_{1}$.  Moreover, $G_{1}$ is a one dimensional torus (it is $C^{\times}$ in the representation given by $f$). Lemma \ref{lem:torus isogeny} implies that $G$ itself is a one dimensional torus. 

Therefore, there exists a basis of solutions $\widetilde{g}_{1},\ldots,\widetilde{g}_{r}$ of $L$ in $E$ such that the image of the representation of the Galois group $G$  with respect to 
$$
\widetilde{\mathfrak{Y}}
=
\begin{pmatrix}
 f&0&\cdots &0\\
 0&\widetilde{g}_{1}&\cdots &\widetilde{g}_{r}\\
 0&\widetilde{g}_{1}'&\cdots &\widetilde{g}_{r}'\\
  0&\vdots&\ddots &\vdots\\
   0&\widetilde{g}_{1}^{(r-1)}&\cdots &\widetilde{g}_{r}^{(r-1)}
\end{pmatrix} 
\in \GL_{n}(E)
$$
is given by 
$$
\{\operatorname{diag}(t^{n_{0}},\ldots,t^{n_{m}}) \ \vert \ t \in C^{\times}\}
$$
for some $n_{0},\ldots,n_{m} \in \mathbb{Z}$. 

Therefore, if $\theta$ is such that $f=\theta^{n_{0}}$, then $\widetilde{g}_{i}=a_{i}\theta^{n_{i}}$ for some $a_{i} \in k^{\times}$ (because $(\widetilde{g}_{i}/\theta^{n_{i}})^{n_{0}}$ is fixed by $G$). This concludes the proof since $g$ is in the $C$-span of the $\widetilde{g}_{i}$. 
\end{pproof}

\begin{cor}\label{cor:f eg Pg}
Let $f \in K^{\times}$ be such that 
$$
f'=af
$$ 
for some $a \in k$.  Let $g \in K$ satisfy a nonzero homogeneous linear differential equation over $k$. Assume that
$$
f=P(g)
$$
for some nonconstant polynomial $P$ with coefficients in $k$. Then, there exists $\theta$ in  $K$ such that $\theta'/\theta \in k$
$$
f=\theta^{n} \text{ and } g=c+d\theta
$$ 
for some $n \in \mathbb{Z}_{\geq 0}$ and some $c,d \in k$. 
\end{cor}

\begin{pproof}
If $f$ is algebraic over $k$, $g$ is algebraic over $k$ as well and, hence, belongs to $k$. 

So, we can and will assume that $f$ is transcendental over $k$. Theorem~\ref{prop:torus} ensures that there exists $\theta$ and an integer $n \in \mathbb{Z}$ such that 
$$
f=\theta^{n}
\text{ and }
g
\in k[\theta,\theta^{-1}]. 
$$

Since $k$ is algebraically closed, we may factor $P(Y)$ as a product of linear factors and, hence, the equation $P(g)=f$ can be rewritten as 
\[e_{1} \prod_{i=1}^\ell (g  -d_i) = \theta^n\]
for some $e_{1}$ and $d_i$ in $k$. Since $\theta$ is transcendental over $k$, the unique factorization property of  $k[\theta,\theta^{-1}]$ implies that each factor in the above equality belong to $k \theta^{\mathbb{Z}}$; in particular,  we have for each $i$
$$
g -d_i=e_{i} \theta^{r_i}
$$
for some $e_{i} \in k^{\times}$ and some $r_i \in \mathbb{Z}$. For any pair $r_i,r_j, i\neq j$, we have $d_i - d_j = e_i\theta^{r_i} - e_j\theta^{r_j}$. Since $\theta$ is transcendental, we must have $d_i = d_j := d,  e_i = e_j := e$ and $r_i = r_j:= r$.  So $g = e\theta^{r} + d$ and $e,d\in k$.\\
In particular we have that $n = rn'$ for some  $n' \in \mathbb{Z}_{\geq 0}$.  Whence the desired result (the $\theta$ in the statement of the Corollary is the $\theta^{r}$ of the proof, and the $n$ of the statement is the $n'$ of the proof).   
\end{pproof}

\begin{cor}\label{cor f dans k de g}
Let $f,g \in K$ satisfy some nonzero homogeneous linear differential equations over $k$. Assume that
$$
f \in k(g). 
$$
Then,
\begin{itemize}
\item either  $f \in k[g]$,
\item or there exists $\theta$ in  $K$ such that $\theta'/\theta \in k$ and 
$$
g= c+d \theta \text{ and } f \in k[\theta,\theta^{-1}]
$$ 
for some $c,d \in k$. 
\end{itemize}
\end{cor}

\begin{pproof}
 Consider $R \in k(X)$ such that $f=R(g)$. Let   $A,B \in k[X]$ be such that $\operatorname{gcd}(A,B)=1$ and $R=A/B$. B\'{e}zout's {Identity states} that there exist $U,V \in k[X]$ such that $AU+BV=1$, so $RU+V=1/B$. It follows that $fU(g)+V(g)=R(g)U(g)+V(g)=1/B(g)$ satisfies a nonzero homogeneous linear differential equations over $k$. By  \cite{HarrisSibuya85} (see also \cite{Sperber}, \cite{Singer86}), we have $B(g)'/B(g) \in k$ and hence we can apply Corollary \ref{cor:f eg Pg}: if $B$ is nonconstant then there exists $\theta$ such that 
$$
B(g)=\theta^{n} \text{ and } g=c+d\theta 
$$ 
for some $n \in \mathbb{Z}_{\geq 0}$ and some $c,d \in k$. 

 Of course, if $B$ is constant, then $f \in k[g]$. 
\end{pproof}

\section{The simply connected case}\label{sec:simpcon}

We begin with some facts concerning simply connected\footnote{Here, simply connected means : a connected algebraic group $G$ not admitting any non-trivial isogeny $\phi:H \rightarrow G$ where $H$ is also a connected algebraic group. For semisimple algebraic groups over the field of complex numbers this definition is equivalent to the topological one.}   algebraic groups.

\begin{lem}~\label{lem:notorus} A simply connected  algebraic group $G$ has no torus quotient, that is, there is no nontrivial homomorphism $\phi:G \rightarrow T = (C^*)^n$. In particular, any multiplicative character of $G$ is trivial.  \end{lem}
\begin{proof} From \cite[Chapter XVIII, Proposition 3.2]{Hochschild}, we know that the radical $R(G)$ is unipotent. Let $\phi:G\rightarrow T=(C^*)^n$ be a homomorphism. Since homomorpisms preserve unipotent elements, we must have that $R(G) \subset \Ker(\phi)$. Therefore $\phi$ factors through $H:=G/R(G)$. {Since $H$ is semisimple\footnote{$R(G)$ is the smallest closed normal subgroup $K$ of $G$ such that $G/K$ is semisimple, cf. \cite[Section 19.5 and Theorem 27.5]{Humphreys}}, it coincides with its derived subgroup $H'$. Therefore, $\phi$ takes its values in the derived subgroup $T'$ of $T$. Since $T'$ is trivial, we have that $\phi$ is trivial.} 
\end{proof}

{In the proof of the following proposition, we give two arguments - one based on the Picard-Vessiot approach to the Galois theory and the other based on the tannakian approach. The tannakian approach plays the dominant role in what follows and  this allows the reader to compare the two approaches in this simpler situation.}

\begin{prop} \label{prop sln}
Let {$f \in K^{\times}\setminus k$} satisfy a nonzero homogeneous linear differential equation $L(f)=0$ over $k$ of order $n\geq 1$ and assume that the differential Galois group $G_{L}$ over $k$ of the latter equation is simply connected. Let $g \in K$ satisfy a nonzero homogeneous linear differential equation over $k$. Assume that $f$ and $g$ are algebraically dependent over $k$. 
\begin{enumerate}
\item[1.] If $F$ is a differential field extension of $k(f,g)$ with field of constants $C$ containing a basis of solutions $f_{1}=f,f_{2},\ldots,f_{n}$ of $L$, then 
$$
g \in k[W(f_{1},\ldots,f_{n})]
$$
where $W(f_{1},\ldots,f_{n})$ denotes the wronskian matrix associated to $f_{1},\ldots,f_{n}$.
\item[2.] Assume 
\vspace{.05in}
\begin{center}
(*) $k(f)$ is relatively algebraically closed in  $k(W(f_{1},\ldots,f_{n}))$.
\end{center} 
\vspace{.05in}
Then,
$$
g \in k[f].
$$ 
\end{enumerate}
\end{prop}

\begin{pproof}
1. Let $Y'=A_{L}Y$ be the linear differential system corresponding to $L$. 

Let $M$ be a nonzero linear differential operator with coefficients in $k$ annihilating $g$ and of minimal order $m$ for this property. Let $Y'=A_{M}Y$ be the corresponding linear differential system.

We let $E$ be a Picard-Vessiot extension over $k$ of the differential system  
\begin{equation}\label{diff syst direct sum bis}
 Y'=
\begin{pmatrix}
 A_{L}&0\\
 0&A_{M}
\end{pmatrix}
Y
\end{equation}
containing $f$ and $g$. Let $G$ be the corresponding differential Galois group over $k$. The differential system \eqref{diff syst direct sum bis} has a fundamental solution of the form  
$$
\mathfrak{Y} 
=
\begin{pmatrix}
 \mathfrak{Y}_{L} & 0\\
 0 & \mathfrak{Y}_{M} 
\end{pmatrix} 
\in \GL_{m+n}(E)
$$
where 
$$
\mathfrak{Y}_{L} =
\begin{pmatrix}
 f_{1}&\cdots &f_{n}\\
 f_{1}'&\cdots &f_{n}'\\
  \vdots&\ddots &\vdots\\
  f_{1}^{(n-1)}&\cdots &f_{n}^{(n-1)}
\end{pmatrix} 
\in \GL_{n}(E)
$$
and 
$$ 
\mathfrak{Y}_{M} =
\begin{pmatrix}
 g_{1}&\cdots &g_{m}\\
 g_{1}'&\cdots &g_{m}'\\
  \vdots&\ddots &\vdots\\
  g_{1}^{(m-1)}&\cdots &g_{m}^{(m-1)}
\end{pmatrix} 
\in \GL_{m}(E)
$$
with $f_{1}=f$ and $g_{1}=g$ (thus, $E$ is the field extension of $k$ generated by the entries of  $\mathfrak{Y}_{L}$ and $\mathfrak{Y}_{M}$). 
The field extension $E_{L}$ (resp. $E_{M}$) of $k$ generated by the $f_{i}^{(j)}$ (resp. by the $g_{i}^{(j)}$) is a  Picard-Vessiot extension over $k$ for $Y'=A_{L}Y$ (resp. $Y'=A_{M}Y$); we let $G_{L}$ (resp. $G_{M}$) be the corresponding differential Galois group over $k$. 

By minimality of $M$, we have 
$$
\operatorname{span}_{C} G g
=
\operatorname{span}_{C} G_{M} g
=
\operatorname{span}_{C}\{g_{1},\ldots,g_{m}\}
$$ 
(note that $Gg=G_{M}g$ because the restriction morphism $\pi_{M} : G \rightarrow G_{M}$ is surjective). 
Moreover, $E_{L}$ is left invariant by $G$. These facts, together with the fact that $g$ is algebraic over $E_{L}$ (because $f$ and $g$ are algebraically dependent over $k$ and $f$ is transcendental over $k$), imply that the $g^{(j)}_{i}$ are algebraic over $E_{L}$ as well, {\it i.e.}, $E$ is a finite extension of $E_{L}$. 

Again, the differential Galois theory implies  that the restriction morphism  
$
\pi_{L} : G \rightarrow G_{L}
$
is surjective and hence induces a short exact sequence 
$$
0 \rightarrow H \rightarrow G \xrightarrow[]{\pi_{L}} G_{L} \rightarrow 0. 
$$
The kernel $H$ (which is the group of differential fields automorphisms of $E$ over $E_{L}$) is finite because $E$ is a finite extension of $E_{L}$.  

Since $G_{L}$ is a connected and simply connected algebraic group, the kernel $H$ is trivial and 
$
\pi_{L} : G \rightarrow G_{L}
$
is an isomorphism. {In particular, since $H$ is trivial, the Galois correspondence implies that $E =E_L$. The $C$-span of the entries of $\mathfrak{Y}_{M}$ form a $G = G_L$-invariant finite dimensional $C$-vector space. It follows from \cite[Corollary 1.38]{VdPS} that all of these entries, and in particular $g = g_1$, lie in the Picard-Vessiot ring $k[\mathfrak{Y}_L, \det(\mathfrak{Y}_L)^{-1}]$.} 
Note that $G_L$ leaves the $C$-line spanned by $\det(\mathfrak{Y}_{L})$ invariant and, since all characters of $G_L$ are trivial {(see Lemma \ref{lem:notorus})}, it must leave $\det(\mathfrak{Y}_{L})$ fixed. Therefore $\det(\mathfrak{Y}_{L}) \in k$ and so $k[\mathfrak{Y}_{L},\det(\mathfrak{Y}_{L})^{-1}]= k[\mathfrak{Y}_{L}]$. This proves the first claim of the Proposition.

{We now turn to a tannakian proof of the first claim. As already noted the fact that $G_{L}$ is a connected and simply connected algebraic group implies that the kernel $H$ is trivial and 
$
\pi_{L} : G \rightarrow G_{L}
$
is an isomorphism. It follows that there exists an algebraic group morphism $u : G_{L} \rightarrow G_{M}$ such that 
$$
G =
\operatorname{im} (\pi_{L} \oplus \pi_{M})
= \operatorname{im} (\operatorname{id}_{G_{L}} \oplus u).
$$

We now adopt the tannakian point of view (see Appendix~\ref{tannaka}):  the category of rational linear representations of $G$ is equivalent to the tannakian category generated by the differential system \eqref{diff syst direct sum bis}. In this equivalence, the differential system $Y'=A_{L}Y$ (resp. $Y'=A_{M}Y$) corresponds to the restriction morphism 
$
\pi_{L} : G \rightarrow G_{L}
$ 
(resp. 
$
\pi_{M} : G \rightarrow G_{M}
$). The factorization  
$$
\pi_{M}=u \circ \pi_{L}
$$ 
ensures that $Y'=A_{M}Y$ belongs to the tannakian category generated by $Y'=A_{L}Y$. It follows that the entries of the solutions of $Y'=A_{M}Y$ belong to the Picard-Vessiot ring $R_{L}=k[\mathfrak{Y}_{L},\det(\mathfrak{Y}_{L})^{-1}]$ of $Y'=A_{L}Y$.    As before $G_L$ leaves the line spanned by $\det(\mathfrak{Y}_{L})$ invariant and, since all characters of $G_L$ are trivial {(see Lemma \ref{lem:notorus})}, it must leave $\det(\mathfrak{Y}_{L})$ fixed. Therefore $\det(\mathfrak{Y}_{L}) \in k$ and so $k[\mathfrak{Y}_{L},\det(\mathfrak{Y}_{L})^{-1}]= k[\mathfrak{Y}_{L}]$. This yields the second proof of the first claim of the Proposition. }

2.{If (*) is satisfied, then $g$ belongs to $k(f)$ (because $g$ is algebraic over $k(f)$).} Corollary~\ref{cor f dans k de g} implies that either  $g \in k[f]$,
or there exists  $\theta$ in  $K$ such that $\theta'/\theta \in k$ and 
$f= c+d \theta \text{ and } g \in k[\theta,\theta^{-1}]$ 
for some $c,d \in k$. {We claim that the latter case is impossible. Indeed, otherwise the line $C \theta$ spanned by $\theta$ would be left   invariant by $G_{L}$.  But Lemma~\ref{lem:notorus} states that any character of $G_L$ is trivial. So $G_{L}$ would act as the identity on $\theta$. By Galois correspondance, this would imply that $\theta \in k$ and, hence, that $f \in k$. This is excluded by the assumption.}  
\end{pproof}

In order to draw the conclusion of  Proposition~\ref{prop case 2}.2, one needs to verify condition (*). We will give two examples of this below but first mention a general approach.  

Since we are assuming $k$ is algebraically closed, the Galois group $G$ of any Picard-Vessiot extension of $k$ is connected and the Picard-Vessiot ring of any linear differential equation over $k$ is $k$-isomorphic to the coordinate ring  $k[G\otimes k]$ of the group $G$ over $k$ (\cite[Corollary 5.16]{Ma}, \cite[Proposition 1.31, Corollary 1.32]{VdPS}). In particular, this implies that the Picard-Vessiot field $K$ is isomorphic to the function field of $G$.  This latter field is known to be a purely transcendental extension of $k$ (\cite[Remark 14.14]{Borel}). Therefore, if one can find algebraically independent elements $z_1 = f, z_2, \ldots , z_m$ such that $K = k(z_1, \ldots , z_m)$ then $k(f)$ will be relatively algebraically closed in $K$.  Note that in Proposition~\ref{prop case 2} the field $k(W(f_1, \ldots , f_n))$ is the Picard-Vessiot field of the equation $L(y) = 0$.

\subsection{When $G_{L}=\SL_n(C)$}

In this section, we will prove 
 \begin{prop} \label{prop  case 2} Let $f \in K^{\times} \setminus k$ satisfy an $n^{th}$ order nonzero homogeneous linear differential equation $L(f)=0$ over $k$ and assume that the differential Galois group $G_{L}$ over $k$ of this equation is $\SL_{n}(C)$ seen as a subgroup of $\GL_n(C)$. Let $g \in K$ satisfy a nonzero homogeneous linear differential equation over $k$. Assume that $f$ and $g$ are algebraically dependent over $k$.  Then, 
$$
g \in k[f].
$$
\end{prop}

The proof will depend on the following Lemma.
\begin{lem}\label{lem sln} Let $K, f, L, G_L, g$ be as in Proposition~\ref{prop sln}.  Assume that $G_{L}$, seen as a subgroup of $\GL_{n}(C)$ via the basis of solutions $f_{1},\ldots,f_{n}$,  also satisfies the following property~: if we let 
$$
k[G_{L} \otimes k]=k[(X_{i,j})_{1 \leq i,j \leq n},\det(X_{i,j})_{1 \leq i \leq n}^{-1}]/I
$$ 
be the coordinate ring of $G_{L} \otimes k$, then, for any $a \in \operatorname{span}_{k} \{X_{i,1} \ \vert \ 1 \leq i,j \leq n \} \mod I$, $k(a)$ is relatively algebraically closed in $k(G_{L} \otimes k)$.  Then, 
$$
g \in k[f].
$$ 
Furthermore, this conclusion holds  if its hypothesis is replaced by the following: assume that $G_{L}$, seen as a subgroup of $\GL_{n}(C)$ via some basis of solutions of $L$, satisfies the following property~:  for any $a \in \operatorname{span}_{k} \{X_{i,j} \ \vert \ 1 \leq i,j \leq n \} \mod I$, $k(a)$ is relatively algebraically closed in $k(G_{L} \otimes k)$. 
\end{lem}
\begin{pproof}  We know that the Picard-Vessiot ring $R_{L}$ of $L$  is isomorphic, as a $k$-algebra, to $k[G_{L} \otimes k]$, more precisely (c.f. \cite[Corollary 5.16]{Ma}, \cite[Proposition 1.31, Corollary 1.32]{VdPS}) there exists $v \in \GL_{n}(k)$ such that the $k$-algebra morphism 
$$
k[(X_{i,j})_{1 \leq i,j \leq n},\det((X_{i,j})_{1 \leq i,j \leq n})^{-1}] \rightarrow R_{L}$$
defined by $ X \mapsto v\mathfrak{Y}_{L}
$  
induces an isomorphism 
$$
k[G_{L} \otimes k]=k[(X_{i,j})_{1 \leq i,j \leq n},\det(X_{i,j})_{1 \leq i,j \leq n}^{-1}]/I \xrightarrow[]{\sim} R_{L}.
$$ 
In this isomorphism, $f$ corresponds to an element of $\operatorname{span}_{k} \{X_{i,1} \ \vert \ 1 \leq i \leq n \} \mod I$, so, the hypothesis on $k(G_{L} \otimes k)$ implies $k(f)$ is relatively algebraically closed in $k(W(f_{1},\ldots,f_{n}))$. Proposition~\ref{prop sln} implies the conclusion. 

The final assertion  follows from this. Indeed, the differential Galois group of $L$ seen as a group of matrices via some basis of solutions of $L$ is conjugate via some element of $\GL_{n}(C)$ to the same differential Galois group seen as a group of matrices via the basis of solutions $f_{1},\ldots,f_{n}$. It follows that if the final hypothesis  is satisfied then the  previous hypothesis is satisfied as well. \end{pproof}

\begin{remark}
 The difference between the two parts of Lemma~\ref{lem sln} is that, in the first part, $G_{L}$ is seen as a group of matrices via the basis of solutions $f_{1},\ldots,f_{n}$ which has $f_{1}=f$ as first element, whereas, in the second, $f$ is not required to be an element of the basis used to see $G_{L}$ as a group of matrices. 
\end{remark}

\begin{remark}
The hypotheses of Lemma~\ref{lem sln} concerning the coordinate ring of $G_{L}\otimes k$ are not hypotheses on the differential Galois group of $L$ as an ``abstract'' algebraic group but hypotheses on the incarnation $G_{L}$ of this differential Galois group as a subgroup of $\GL_{n}(C)$ {\it via} the linear representation given by the choice of a basis of solutions of $L$. In other words, these hypotheses are not necessarily invariant by isomorphism of algebraic groups. For instance, this hypothesis  is satisfied by $G_{L}=\SL_{2}(C) \subset \GL_2(C)$, but not by $\operatorname{Sym}^{3}(\operatorname{SL}_{2}(C)) \subset \GL_4(C)$, despite the fact that these two algebraic groups are isomorphic. The fact that the hypothesis  is satisfied when $G_{L}=\SL_{2}(C)$ is shown in the proof of {Proposition~\ref{prop  case 2}} below. Let us prove that it is not satisfied by $G_{L}=\operatorname{Sym}^{3}(\operatorname{SL}_{2}(C))$.  We have denoted by $\operatorname{Sym}^{3}(\operatorname{SL}_{2}(C))$ the image of $\SL_{2}(C)$ by its (faithful) $3$rd symmetric power representation
$$
 \operatorname{Sym}^{3} :  
 \SL_{2}(C)
 \rightarrow
\GL_{4}(C), 
 $$
which has the form 
$$
  \begin{pmatrix}
 a & b \\
 c & d 
\end{pmatrix}
 \mapsto 
  \begin{pmatrix}
 a^{3} & 3 a^{2}b &3 ab^{2}& b^{3} \\
a^{2}c  & * &*& b^{2}d \\ 
 ac^{2} & * &*& bd^{2}\\
  c^{3} & 3 c^{2}d &3 cd^{2}& d^{3}
\end{pmatrix}.
 $$
Let 
\begin{multline*}
 k[\operatorname{Sym}^{3}(\operatorname{SL}_{2}(C)) \otimes k]=k[(X_{i,j})_{1 \leq i,j \leq 3},\det(X_{i,j})_{1 \leq i,j \leq 3}^{-1}]/I \\ 
 =
k[(x_{i,j})_{1 \leq i,j \leq 3},\det(x_{i,j})_{1 \leq i,j \leq 3}^{-1}]
\end{multline*}
be the coordinate ring of $\operatorname{Sym}^{3}(\operatorname{SL}_{2}(C))$.
Then,  
$$
f=3^{-1}x_{1,2}(x_{2,4}/x_{1,4} - x_{2,1}/x_{1,1}) \in k(\operatorname{Sym}^{3}(\operatorname{SL}_{2}(C)) \otimes k)
$$
satisfies $f^{3}=x_{1,1}$. Therefore, $f$ is algebraic over $k(x_{1,1})$ but does not belong to $k(x_{1,1})$ and, hence, $k(x_{1,1})$ is not relatively algebraically closed in $k(\operatorname{Sym}^{3}(\operatorname{SL}_{2}(C)) \otimes k)$.

This is not a surprise: the conclusion of  Proposition \ref{prop case 2} is not satisfied when $G_{L}=\operatorname{Sym}^{3}(\operatorname{SL}_{2}(C))$. Indeed, let $M$ be a linear differential equation over $k$ of order $2$ with differential Galois group $\SL_{2}(C)$. Consider the third symmetric power $L=M^{\circledS 3}$; this is a linear differential equation over $k$ of order $4$ with differential Galois group $G_{L}=\operatorname{Sym}^{3}(\operatorname{SL}_{2}(C))$. If $(u,v)$ is a basis of solutions of $M$, then $(u^{3},u^{2}v,uv^{2},v^{3})$ is a basis of solutions of $L$. Therefore, $f=u^{3}$ is a solution of $L$, $g=u$ is a solution of $M$, $f \not \in k$ and $f$ and $g$ are algebraically dependent over $k$ but $g \not \in k(f)$. 
 \end{remark}

\begin{proofsln} We shall prove that the hypotheses  of Lemma~\ref{lem sln} are satisfied when $G_{L}=\SL_{n}(C)$. The fact that $\SL_{n}(C)$ is connected and simply connected is well-known (cf. \cite[Section 31.1]{Humphreys}). It remains to prove that, if we denote by 
\begin{multline*}
 k[\SL_{n}(C) \otimes k] 
=
k[(X_{i,j})_{1 \leq i,j \leq n},\det(X_{i,j})_{1 \leq i,j \leq n}^{-1}]/(\det (X_{i,j})_{1 \leq i,j \leq n} -1) \\
=
k[(x_{i,j})_{1 \leq i,j \leq n}]
\end{multline*}
the coordinate ring of $\SL_{n}(C) \otimes k$, we have: for any 
$$
a = \sum_{1 \leq i,j \leq n} \lambda_{i,j} x_{i,j} \in \operatorname{span}_{k} \{x_{i,j} \ \vert \ 1 \leq i,j \leq n \},
$$ 
$k(a)$ is relatively algebraically closed in $k(\SL_{n}(C) \otimes k)$. 

Let us first assume that at least one of the $\lambda_{i,j}$ is zero, say $\lambda_{i_{0},j_{0}} = 0$. Then, we have 
\begin{multline*}
k(\SL_{n}(C) \otimes k)=
k((x_{i,j})_{1 \leq i,j \leq n})\\
=
k((x_{i,j})_{1 \leq i,j \leq n, (i,j) \neq (i_{0},j_{0})}) 
=
k(a,(x_{i,j})_{1 \leq i,j \leq n, (i,j) \neq (i_{0},j_{0}), (i_{1},j_{1})})\end{multline*}
where $(i_{1},j_{1})$ is chosen such that $\lambda_{i_{1},j_{1}} \neq 0$. The second equality above follows from that fact $x_{i_{0},j_{0}} \in k((x_{i,j})_{1 \leq i,j \leq n, (i,j) \neq (i_{0},j_{0})})$ which itself follows from the equality $\det (x_{i,j})_{1 \leq i,j \leq n}=1$.  The third equality follows from the fact that $x_{i_{1},j_{1}}$ belongs to the $k$-span of $a$ and $(x_{i,j})_{1 \leq i,j \leq n, (i,j) \neq (i_{0},j_{0}), (i_{1},j_{1})}$. 
Since 
$$
\operatorname{tr.deg}(k(\SL_{n}(C) \otimes k)/k)=n^{2}-1 $$ 
and 
$$
\sharp \{a,(x_{i,j})_{1 \leq i,j \leq n, (i,j) \neq (i_{0},j_{0}), (i_{1},j_{1})}\} \leq n^{2}-1,
$$ 
we get that the family $(a,(x_{i,j})_{1 \leq i,j \leq n, (i,j) \neq (i_{0},j_{0}), (i_{1},j_{1})})$ is algebraically independent over $k$. In particular, $k(a)$ is algebraically closed in  $k(\SL_{n}(C) \otimes k)$.

Let us now assume that the $\lambda_{i,j}$ are nonzero. We have 
$$
k(\SL_{n}(C) \otimes k)=
k((x_{i,j})_{1 \leq i,j \leq n})
=
k(a,(x_{i,j})_{1 \leq i,j \leq n, (i,j) \neq (1,1)}).
$$ 
The latter equality follows from the fact that $x_{1,1}$ belongs to the $k$-span of $a$ and $(x_{i,j})_{1 \leq i,j \leq n, (i,j) \neq (1,1)}$.
 But, we have 
 $$
 x_{2,1}= \lambda_{2,1}^{-1} a - \sum_{\substack{1 \leq i,j \leq n \\ (i,j) \neq (1,1),(2,1)}}  \lambda_{2,1}^{-1} \lambda_{i,j} x_{i,j} - \lambda_{2,1}^{-1} \lambda_{1,1} x_{1,1}. 
 $$
 The equality $\det (x_{i,j})_{1 \leq i,j \leq n}=1$ ensures that 
 $$
 x_{1,1} = -\frac{\delta_{2,1}}{\delta_{1,1}} x_{2,1} - \sum_{i=3}^{n}  \frac{\delta_{i,1}}{\delta_{1,1}} x_{i,1} + \frac{1}{\delta_{1,1}}
 $$
 where $\delta_{i,j}$ is the $(i,j)$ cofactor of the matrix $(x_{i,j})_{1 \leq i,j \leq n}$. 
 It follows that 
 \begin{multline*}
 \left(1 - \lambda_{2,1}^{-1} \lambda_{1,1}\frac{\delta_{2,1}}{\delta_{1,1}}\right)x_{2,1} = \\ \lambda_{2,1}^{-1} a - \sum_{\substack{1 \leq i,j \leq n \\ (i,j) \neq (1,1),(2,1)}}  \lambda_{2,1}^{-1} \lambda_{i,j} x_{i,j} + \lambda_{2,1}^{-1} \lambda_{1,1}\left(\sum_{i=3}^{n}  \frac{\delta_{i,1}}{\delta_{1,1}} x_{i,1} - \frac{1}{\delta_{1,1}}\right). 
\end{multline*}
 This shows that $x_{2,1}$ belongs to $k(a,(x_{i,j})_{1 \leq i,j \leq n, (i,j) \neq (1,1),(2,1)})$ and, hence, that 
 $$
k(\SL_{n}(C) \otimes k)=
k(a,(x_{i,j})_{1 \leq i,j \leq n, (i,j) \neq (1,1),(1,2)}).
$$ 
 We can now conclude that $k(a)$ is algebraically closed in  $k(\SL_{n}(C) \otimes k)$ by using a transcendence degree argument as we did above. \end{proofsln}

\section{A special case of reductive $G_L$: $\GL_n(C)$} \label{sec:specialcase}

\begin{lem}\label{lem isogeny}
 Let $\varphi : G_{1} \rightarrow G_{2}$ be an isogeny with $G_{1}$ connected. Then: 
 \begin{itemize}
 \item $\varphi^{-1}(Z_{G_{2}})=Z_{G_{1}}$, $\varphi(Z_{G_{1}})=Z_{G_{2}}$ and $\varphi(Z_{G_{1}}^{0})=Z_{G_{2}}^{0}$; 
   \item if $G_{2}=Z_{G_{2}}^{0}G_{2}'$, then $G_{1}=Z_{G_{1}}^{0}G_{1}'$\footnote{Recall that $G'$ is the derived subgroup of the algebraic group $G$.}; 
  \item if $Z_{G_{2}}^{0}$ is a torus of dimension $d$, then $Z_{G_{1}}^{0}$ is a torus of dimension $d$ as well; 
\item if $G_{2}'$ is semi-simple, then $G_{1}'$ is semi-simple as well;
\item if $G_{2}$ is reductive, then $G_{1}$ is reductive as well. 
\end{itemize}
 \end{lem}

\begin{pproof}
We first prove that $ \varphi^{-1}(Z_{G_{2}}) \subset Z_{G_{1}}$. Consider $g \in \varphi^{-1}(Z_{G_{2}})$. For any $h \in G_{1}$, $\varphi([g,h])=[\varphi(g),\varphi(h)]=e_{G_{2}}$, so $[g,h] \in \ker \varphi$. Since $G_{1}$ is connected, the image the morphism of algebraic varieties 
\begin{eqnarray*}
 G_{1} &\rightarrow& \ker \varphi\\
h &\mapsto & [g,h]
\end{eqnarray*}
is contained in $(\ker \varphi)^{0}=\{e_{G_{1}}\}$, so $g \in Z_{G_{1}}$ as expected. 

The inclusion $Z_{G_{1}} \subset \varphi^{-1}(Z_{G_{2}})$ follows from the fact that $\varphi$ is onto and therefore $\varphi^{-1}(Z_{G_{2}}) =Z_{G_{1}}$. 

The equality $\varphi(Z_{G_{1}})=Z_{G_{2}}$ follows form this and from the fact that $\varphi$ is onto. 

It follows that $\varphi(Z_{G_{1}}^{0})=\varphi(Z_{G_{1}})^{0}=Z_{G_{2}}^{0}$.

Since $\varphi$ is onto, we have $\varphi(G_{1}')= G_{2}'$. We have seen that $\varphi(Z_{G_{1}}^{0})=Z_{G_{2}}^{0}$. Therefore, $\varphi(Z_{G_{1}}^{0}G_{1}')=Z_{G_{2}}^{0} G_{2}'=G_{2}$. Since $\ker \varphi$ is central, we get $G_{1} \subset (\ker \varphi)Z_{G_{1}}^{0}G_{1}' \subset Z_{G_{1}}G_{1}'$, so $G_{1} = Z_{G_{1}}G_{1}'$.  But $Z_{G_{1}}^{0}G_{1}'$ is a closed subgroup of $G_{1}=Z_{G_{1}}G_{1}'$ of finite index, so $G_{1}^{0}=G_{1} \subset Z_{G_{1}}^{0}G_{1}'$, whence the equality $G_{1} = Z_{G_{1}}^{0}G_{1}'$.

Assume that $Z_{G_{2}}^{0}$ is a torus of dimension $d$. We have seen that $\varphi(Z_{G_{1}}^{0})=Z_{G_{2}}^{0}$, so $\varphi$ induces an isogeny between $Z_{G_{1}}^{0}$ and $Z_{G_{2}}^{0}$. Lemma \ref{lem:torus isogeny} implies that $Z_{G_{1}}^{0}$ is a torus of dimension $d$. 

Since $\varphi(G_{1}')= G_{2}'$, $\varphi$ induces an isogeny between $G_{1}'$ and $G_{2}'$. Therefore, $G_{1}'$ is semi-simple if and only if $G_{2}'$ is semi-simple.  

The last assertion follows from the previous one. 
\end{pproof}

{In the proof of the next result we will use the tannakian correspondence. Appendix ~\ref{tannaka} contains a review of this tool.}

\begin{prop} \label{prop: case GL}
Let $f \in K^{\times}$ satisfy a nonzero homogeneous linear differential equation over $k$ of order $n$ and assume that the differential Galois group over $k$ of the latter equation is $\GL_{n}(C)$.  Let $g \in K$ satisfy a nonzero homogeneous linear differential equation over $k$. Assume that $f$ and $g$ are algebraically dependent over $k$. We have:
\begin{enumerate}
\item if $n=1$, then there exist $\theta$ and a relative integer $n$ such that 
$$
f=\theta^{n}
\text{ and }
g \in k[\theta,\theta^{-1}];
$$ 
\item if $n>1$, then 
$$
g \in k[f].
$$

\end{enumerate}
\end{prop}

\begin{pproof} 
Let $L$ be a nonzero linear differential operator of order $n$ with coefficients in $k$ and with differential Galois group over $k$ equal to $\GL_{n}(C)$ annihilating $f$. Let $Y'=AY$ be the corresponding linear differential system.

Let $M$ be a nonzero linear differential operator with coefficients in $k$ annihilating $g$ and of minimal order $m$ for this property. Let $Y'=BY$ be the corresponding linear differential system.

We let $E$ be a Picard-Vessiot extension over $k$ of the differential system  
\begin{equation}\label{diff syst direct sum quat}
 Y'=
\begin{pmatrix}
 A&0\\
 0&B
\end{pmatrix}
Y
\end{equation}
containing $f$ and $g$. Let $G$ be the corresponding differential Galois group over $k$. The differential system \eqref{diff syst direct sum quat} has a fundamental solution of the form  
$$
\mathfrak{Y} 
=
\begin{pmatrix}
 \mathfrak{Y}_{A} & 0\\
 0 & \mathfrak{Y}_{B} 
\end{pmatrix} 
\in \GL_{m+n}(E)
$$
where 
$$
\mathfrak{Y}_{A} =
\begin{pmatrix}
 f_{1}&\cdots &f_{n}\\
 f_{1}'&\cdots &f_{n}'\\
  \vdots&\ddots &\vdots\\
  f_{1}^{(n-1)}&\cdots &f_{n}^{(n-1)}
\end{pmatrix} 
\in \GL_{n}(E)
$$
and 
$$
\mathfrak{Y}_{B} =
\begin{pmatrix}
 g_{1}&\cdots &g_{m}\\
 g_{1}'&\cdots &g_{m}'\\
  \vdots&\ddots &\vdots\\
  g_{1}^{(m-1)}&\cdots &g_{m}^{(m-1)}
\end{pmatrix} 
\in \GL_{m}(E)
$$
with $f_{1}=f$ and $g_{1}=g$ (thus, $E$ is the field extension of $k$ generated by the entries of  $\mathfrak{Y}_{A}$ and $\mathfrak{Y}_{B}$). 
The field extension $E_{A}$ (resp. $E_{B}$) of $k$ generated by the $f_{i}^{(j)}$ (resp. by the $g_{i}^{(j)}$) is a  Picard-Vessiot extension over $k$ for $Y'=AY$ (resp. $Y'=BY$); we let $G_{A}$ (resp. $G_{B}$) be the corresponding differential Galois group over $k$. 

By minimality of $M$, we have 
$$
\operatorname{span}_{C} G g
=
\operatorname{span}_{C} G_{B} g
=
\operatorname{span}_{C}\{g_{1},\ldots,g_{m}\}
$$ 
(note that $Gg=G_{B}g$ because the restriction morphism $G \rightarrow G_{B}$ is surjective). 
Moreover, $E_{A}$ is left invariant by $G$. These facts, together with the fact that $g$ is algebraic over $E_{A}$ (because $f$ and $g$ are algebraically dependent over $k$ and $f$ is transcendental over $k$ because {$G_A$ is $\GL_n(C)$})
imply that the $g^{(j)}_{i}$ are algebraic over $E_{A}$ as well, {\it i.e.}, $E$ is a finite extension of $E_{A}$. 

It is a general fact that the restriction morphism  
$
\pi_{A} : G \rightarrow G_{A}
$
is surjective and hence induces a short exact sequence 
$$
0 \rightarrow H \rightarrow G \xrightarrow[]{\pi_{A}} G_{A} \rightarrow 0. 
$$
The kernel $H$ (which is the group of differential field automorphisms of $E$ over $E_{A}$) is finite because $E$ is a finite extension of $E_{A}$.  
Since $k$ is algebraically closed, $G$ is connected. Applying Lemma \ref{lem isogeny} to $\pi_{A}:G \rightarrow G_{A}$ (with $G_{A}=\GL_{n}(C)=Z_{G_{A}}^{0}G_{A}'$ with $Z_{G_{A}}^{0}=C^{\times}I_{n}$ and $G_{A}'=\SL_{n}(C)$), we get 
$$
G=Z_{G} G'=Z_{G}^{0} G'
$$ 
and also that $Z_{G}^{0}$ is a one dimensional torus. 
There exists $P\in \operatorname{diag}(I_{n},\GL_{m}(C)) \subset \GL_{n+m}(C)$ such that 
$$
P Z_{G}^{0} P^{-1}= 
\{\operatorname{diag}(t^{\alpha_{0}}I_{n},t^{\alpha_{1}}I_{m_{1}},\ldots,t^{\alpha_{\ell}}I_{m_{\ell}}) \ \vert \ t \in C^{\times}\}
$$
for some $m_{1},\ldots,m_{\ell} \in \mathbb{Z}_{\geq 1}$ and some $\alpha_{0},\ldots,\alpha_{\ell} \in \mathbb{Z}$ with $\alpha_{1},\ldots,\alpha_{\ell}$ two by two distinct.  Therefore, $PGP^{-1}$ is diagonal by block: 
$$
PGP^{-1} \subset 
\operatorname{diag}(\GL_{n}(C),\GL_{m_{1}}(C),\ldots,\GL_{m_{\ell}}(C)).
$$
It follows that there exist $B_{1} \in M_{m_{1}}(k),\ldots,B_{\ell}  \in M_{m_{\ell}}(k)$ such that $Y'=BY$ is equivalent over $k$ to 
$$
Y'=
 \text{diag}(B_{1},\ldots,B_{\ell})
Y.
$$
So, up to changing $\mathfrak{Y}_{B}$ by $R\mathfrak{Y}_{B}$ for some $R \in \GL_{m}(k)$, we can and will assume that $B= \text{diag}(B_{1},\ldots,B_{\ell})$ and that $P=I_{n}$.

Let $\widetilde{A}=A-\operatorname{tr}(A)/nI_n$ and consider the differential system 
\begin{align} \label{eq:split}
 Y'=&
 \text{diag}(\operatorname{tr}(A)/nI_n,\widetilde{A},B_{1},\ldots,B_{\ell})
Y. 
\end{align}
Since $\operatorname{tr}(\widetilde{A}) = 0$, the differential Galois group of $\widetilde{A}$ is unimodular (\cite[Exercise 1.35.5]{VdPS}).
We let $H$ be the differential Galois group of \eqref{eq:split}. In the tannakian correspondance (cf., Example~\ref{ex:appendix}), 
the equation 
$$
 Y'=
 \text{diag}(A,B_{1},\ldots,B_{\ell})
Y
$$
 corresponds to the representation $\rho: H \rightarrow G$ given by 
$$
\rho:  \text{diag}(t,\widetilde{h},h_{1},\ldots,h_{l}) \mapsto  \text{diag}(t \widetilde{h},h_{1},\ldots,h_{\ell}). 
$$
Applying Lemma \ref{lem isogeny} to $\rho: H \rightarrow G$ ({again the differential Galois theory implies  that $\rho$ is onto, and  that it has finite kernel}), we have 
$$
H=Z_{H}^{0} H'
$$
and $Z_{H}^{0}=\rho^{-1}(Z_{G}^{0})^{0} $ is a one dimensional torus. 
The restriction  
$$
\rho_{\vert Z_{H}^{0}}: Z_{H}^{0}=\rho^{-1}(Z_{G}^{0})^{0} \rightarrow Z_{G}^{0}
$$
is given by 
$$
h \mapsto (\chi_{0}(h)I_{n},\chi_{1}(h)I_{m_{1}},\ldots,\chi_{\ell}(h)I_{m_{\ell}})
$$
for some characters $\chi_{i}: Z_{H}^{0} \rightarrow C^{\times}$. Another morphism 
$
Z_{H}^{0} \rightarrow Z_{G}^{0}
$
is given by
$$
h=\operatorname{diag}(t,\widetilde{h},h_{1},\ldots,h_{\ell}) \mapsto \operatorname{diag}(t^{\alpha_{0}}I_{n},t^{\alpha_{1}}I_{m_{1}},\ldots,t^{\alpha_{\ell}}I_{m_{\ell}}).
$$
Therefore (since $Z_{K}^{0}$ and $Z_{G}^{0}$ are one dimensional tori: see Lemma~\ref{lem : isogenies tori}), there exist nonzero relatively prime integers $N,M$ such that, for all $i$,
$$
\chi_{i}(h)^{N}=t^{\alpha_{i}M}.
$$

Consider the differential system  
$$
 Y'=
 \text{diag}(\operatorname{tr}(A)/nN,\widetilde{A},B_{1},\ldots,B_{\ell}) 
Y.
$$
Let $H_{N}$ be the differential Galois group of this system. In the tannakian correspondance, 
the equation 
$$
 Y'=
 \text{diag}(A,B_{1},\ldots,B_{\ell})
Y
$$
 corresponds to 
 $$
 \varrho := \rho \circ p:H_{N} \rightarrow G
 $$ 
 where $p: H_{N} \rightarrow H$ (which corresponds to $Y'=
 \text{diag}(\operatorname{tr}(A)/n,\widetilde{A},B_{1},\ldots,B_{\ell})
Y$ in the tannakian correspondence) is given by
$$
p: \operatorname{diag}(z,\widetilde{h},h_{1},\ldots,h_{l}) \mapsto \operatorname{diag}(z^{N},\widetilde{h},h_{1},\ldots,h_{\ell}). 
$$
Applying Lemma \ref{lem isogeny} to $p: H_{N} \rightarrow H$, we get 
$$
H_{N}=Z_{H_{N}}^{0} H_{N}'
$$
and $Z_{H_{N}}^{0}=p^{-1}(Z_{H}^{0})^{0}$ is a one dimensional torus. 

If $h=\operatorname{diag}(z,\widetilde{h},h_{1},\ldots,h_{\ell}) \in Z_{H_{N}}^{0}=p^{-1}(Z_{H}^{0})^{0}$, then 
$$
\chi_{i}(p(h))^{N}=z^{\alpha_{i}NM}
$$
so 
$$
\chi_{i}(p(h))=z^{\alpha_{i}M}.
$$
So, we have 
$$
\varrho_{\vert Z_{H_{N}}^{0}}: (z,\widetilde{h},h_{1},\ldots,h_{\ell}) \mapsto (z^{N}I_{n}=z^{\alpha_{0}M}I_{n},z^{\alpha_{1}M}I_{m_{1}},\ldots,z^{\alpha_{\ell}M}I_{m_{\ell}}). 
$$

On the other hand, the representation corresponding to $Y'=AY$ in the tannakian correspondence is   
$$
\varpi_{A} = \pi_{A} \circ \varrho: H_{N} \rightarrow G_{A}.
$$
It is surjective ({differential Galois theory}) and has finite kernel (because $\varrho$ and $\pi_{A}$ have finite kernels), so it is an isogeny. Therefore, it induces an isogeny 
$$
\varpi_{A\vert H_{N}'} = \pi_{A} \circ \varrho: H_{N}'\rightarrow G_{A}'.
$$
Since $G_{A}'=\SL_{n}(C)$ is a connected and simply connected algebraic group, 
$$
\varpi_{A\vert H_{N}'} = \pi_{A} \circ \varrho: H_{N}'\rightarrow G_{A}'.
$$
is actually an isomorphism. Since $G=\operatorname{im}(\pi_{A} \circ \varrho \oplus \pi_{B} \circ \varrho)$, it follows that there exists an algebraic group morphism $u=(u_{1},\ldots,u_{\ell}): G_{A}' \rightarrow G_{B}$ such that 
$$
\pi_{B} \circ \varrho= u \circ (\pi_{A} \circ \varrho).
$$
So, 
$$
\varrho_{\vert H_{N}'}: 
 (1,\widetilde{h},h_{1},\ldots,h_{\ell}) \mapsto (\widetilde{h},u_{1}(\widetilde{h}),\ldots,u_{\ell}(\widetilde{h})).
$$

Finally, we see that the representation $\varpi_{B} =\pi_{B} \circ \varrho$ corresponding to $Y'=BY$ in the tannakian correspondence is   
given by 
$$
\varpi_{B} =\pi_{B} \circ \varrho: (z,\widetilde{h},h_{1},\ldots,h_{l}) \mapsto (z^{\alpha_{1}M}u_{1}(\widetilde{h}),\ldots,z^{\alpha_{\ell}M}u_{\ell}(\widetilde{h})).
$$

This ensures, by the tannakian correspondance, that $Y'=B_{i}Y$ is equivalent to a system of the form $Y'=(\operatorname{Constr}_{i}(\widetilde{A})+\alpha_{i}M(\operatorname{tr}(A)/nN)I_{m_{i}})Y$ where $\operatorname{Constr}_{i}(\widetilde{A})$ is obtained from $\widetilde{A}$ by some construction of linear algebra. It follows that the entires of the solutions of $Y'=BY$ belong to  $k[\mathfrak{Y}_{L},\sqrt[q]{\det(\mathfrak{Y}_{L})}^{-1}]$  for some positive integer $q$; in particular, $g$ belongs to $k[\mathfrak{Y}_{L},\sqrt[q]{\det(\mathfrak{Y}_{L})}^{-1}]$. 

For $n=1$, this gives the desired result {(which, of course, already follows from Corollary~\ref{cor:f eg Pg})}. 

Let us now assume $n>1$. We have that $R_{L}=k[\mathfrak{Y}_{L},{\det(\mathfrak{Y}_{L})}^{-1}]$ is isomorphic, as a $k$-algebra, to $k[\GL_{n}(k)]$. More precisely there exists $v \in \GL_{n}(k)$ such that the $k$-algebra morphism 
$$
k[(X_{i,j})_{1 \leq i,j \leq n},\det((X_{i,j})_{1 \leq i,j \leq n})^{-1}] \rightarrow R_{L}
$$
defined by $ X \mapsto v\mathfrak{Y}_{L}$   
is an isomorphism; it induces an isomorphism 
$$
k[(X_{i,j})_{1 \leq i,j \leq n},\sqrt[q]{\det((X_{i,j})_{1 \leq i,j \leq n})}^{-1}] \rightarrow k[\mathfrak{Y}_{L},\sqrt[q]{\det(\mathfrak{Y}_{L})}^{-1}]
$$
In this isomorphism, $f$ corresponds to an element $a$ of $\operatorname{span}_{k} \{X_{i,1} \ \vert \ 1 \leq i \leq n \} \mod I$. The fact that $k(a)$ is relatively algebraically closed in  $k((X_{i,j})_{1 \leq i,j \leq n},\sqrt[q]{\det((X_{i,j})_{1 \leq i,j \leq n})}^{-1})$ (see Lemma \ref{lem GL sqrt det} {below}) implies that 
$$
g \in k(f).
$$
Corollary \ref{cor f dans k de g} ensures that:
\begin{itemize}
\item either  $g \in k[f]$,
\item or there exists $\theta$ such that $\theta'/\theta \in k$ and 
$$
f= c+d \theta \text{ and } g \in k[\theta,\theta^{-1}]
$$ 
for some $c,d \in k$. 
\end{itemize}
The latter possibility is incompatible with the fact that $G_{L}=\GL_{n}(C)$ acts irreducibly on the solution space of $L$. 
\end{pproof}

\begin{lem}\label{lem GL sqrt det}
 For any $a \in \operatorname{span}_{k} \{X_{i,j} \ \vert \ 1 \leq i,j \leq n \}$, $k(a)$ is relatively algebraically closed in  
 $$
 k((X_{i,j})_{1 \leq i,j \leq n},\Delta^{-1}) \text{ where } \Delta = \sqrt[q]{\det((X_{i,j})_{1 \leq i,j \leq n})}.
 $$
\end{lem}

\begin{pproof}
Let  
$$
L =
k((X_{i,j})_{1 \leq i,j \leq n},\Delta).
$$ 
and let $a=\sum \lambda_{i,j} X_{i,j}$ be in the $k$-span of the $X_{i,j}$. 

Let us first assume at least one of the $\lambda_{i,j}$ is zero, say $\lambda_{i_{0},j_{0}} = 0$. Then, we have 
\begin{multline*}
L=
k((X_{i,j})_{1 \leq i,j \leq n},\Delta)
=
k((X_{i,j})_{1 \leq i,j \leq n, (i,j) \neq (i_{0},j_{0})},\Delta) \\
=
k(a,(X_{i,j})_{1 \leq i,j \leq n, (i,j) \neq (i_{0},j_{0}), (i_{1},j_{1})},\Delta)
\end{multline*}
where $(i_{1},j_{1})$ is such that $\lambda_{i_{1},j_{1}} \neq 0$. The second equality follows from the fact that $X_{i_{0},j_{0}} \in k((X_{i,j})_{1 \leq i,j \leq n, (i,j) \neq (i_{0},j_{0})},\Delta)$ which itself follows from the equality $$
\det (X_{i,j})_{1 \leq i,j \leq n}=\sum_{j=1}^{n} X_{i_{0},j} \delta_{i_{0},j}=\Delta^{q}
$$
where $\delta_{i,j}$ is the $(i,j)$ cofactor of $(X_{i,j})_{1 \leq i,j \leq n}$. The third equality follows from the fact that $X_{i_{1},j_{1}}$ belongs to the $k$-span of $a$ and $(X_{i,j})_{1 \leq i,j \leq n, (i,j) \neq (i_{0},j_{0}), (i_{1},j_{1})}$. 
Since 
$$
\operatorname{tr.deg}(L/k)=n^{2} \text{ and } \sharp \{a,(X_{i,j})_{1 \leq i,j \leq n, (i,j) \neq (i_{0},j_{0}), (i_{1},j_{1})},\Delta\} \leq n^{2},
$$ 
we get that the family $(a,(X_{i,j})_{1 \leq i,j \leq n, (i,j) \neq (i_{0},j_{0}), (i_{1},j_{1})},\Delta)$ is algebraically independent over $k$. In particular, $k(a)$ is algebraically closed in  $L$.

Let us now assume that the $\lambda_{i,j}$ are nonzero. We have 
$$
L=
k((X_{i,j})_{1 \leq i,j \leq n},\Delta)
=
k(a,(X_{i,j})_{1 \leq i,j \leq n, (i,j) \neq (1,1)},\Delta)\\
$$ 
The latter equality follows from the fact that $X_{1,1}$ belongs to the $k$-span of $a$ and $(X_{i,j})_{1 \leq i,j \leq n, (i,j) \neq (1,1)}$.
{We} have 
 $$
 X_{2,1}= \lambda_{2,1}^{-1} a - \sum_{\substack{1 \leq i,j \leq n \\ (i,j) \neq (1,1),(2,1)}}  \lambda_{2,1}^{-1} \lambda_{i,j} X_{i,j} - \lambda_{2,1}^{-1} \lambda_{1,1} X_{1,1}. 
 $$
 {Furthermore}  the equality $\det (X_{i,j})_{1 \leq i,j \leq n}=\Delta^{q}$ ensures that 
 $$
 X_{1,1} = -\frac{\delta_{2,1}}{\delta_{1,1}} X_{2,1} - \sum_{i=3}^{n}  \frac{\delta_{i,1}}{\delta_{1,1}} X_{i,1} + \frac{\delta^{q}}{\delta_{1,1}}
 $$
 where $\delta_{i,j}$ is the $(i,j)$ cofactor of $(X_{i,j})_{1 \leq i,j \leq n}$. 
 It follows that 
\begin{multline*}
 \left(1 + \lambda_{2,1}^{-1} \lambda_{1,1}\frac{\delta_{2,1}}{\delta_{1,1}}\right)X_{2,1} = \\ \lambda_{2,1}^{-1} a - \sum_{\substack{1 \leq i,j \leq n \\ (i,j) \neq (1,1),(2,1)}}  \lambda_{2,1}^{-1} \lambda_{i,j} X_{i,j} + \lambda_{2,1}^{-1} \lambda_{1,1}\left(\sum_{i=3}^{n}  \frac{\delta_{i,1}}{\delta_{1,1}} X_{i,1} - \frac{\delta^{q}}{\delta_{1,1}}\right). 
\end{multline*}
 This shows that $X_{2,1}$ belongs to $k((X_{i,j})_{1 \leq i,j \leq n, (i,j) \neq (1,1),(2,1)},\Delta)$ and, hence, that 
 $$
L=
k(a,(X_{i,j})_{1 \leq i,j \leq n, (i,j) \neq (1,1),(1,2)},\Delta).
$$ 
 We can now conclude as above by a transcendence degree argument.  
\end{pproof}

\section{The general reductive case}\label{sec:genred}

\begin{lem}\label{lem : isogenies tori}
 Let $T_{1}$ and $T_{2}$ be tori of dimension $d$. Let $\phi,\varphi: T_{1} \rightarrow T_{2}$ be two isogenies. Then, there exist $N \in \mathbb{Z}_{\geq 1}$ and an isogeny $\psi: T_{2} \rightarrow T_{2}$ such that $\varphi^{N}=\psi \circ \phi$.  
\end{lem}

\begin{pproof}
Using the fact that $T_{1}$ and $T_{2}$ are isomorphic to $(C^{\times})^{d}$, we can assume that $T=T_{1}=T_{2}=(C^{\times})^{d}$. We let $X(T)$ be the character group of $T$; it is a free abelian group of rank $d$.  We set $\phi=(\phi_{1},\ldots,\phi_{d})$ and $\varphi=(\varphi_{1},\ldots,\varphi_{d})$; the $\phi_{i}$ and $\varphi_{i}$ are thus elements of $X(T)$. Since $\phi:T \rightarrow T$ is an isogeny, $\phi_{1},\ldots,\phi_{d}$ are multiplicatively independent elements of  $X(T)$ and, hence, they generate a free subgroup $\langle \phi_{1},\ldots,\phi_{d} \rangle$ of rank $d$ of $X(T)$. Thus, every element of $X(T)$ has a positive power belonging to $\langle \phi_{1},\ldots,\phi_{d} \rangle$. In particular, there exist a positive integer $N$ such that the $\varphi_{i}^{N}$ belong to $\langle \phi_{1},\ldots,\phi_{d} \rangle$. Therefore, there exists $\psi_{i} \in X(T)$ such that $\varphi_{i}^{N}=\psi_{i} \circ \phi$. Thus, the morphism $\psi =(\psi_{1},\ldots,\psi_{d}): T \rightarrow T$ is such that $\varphi^{N}=\psi \circ \phi$.  Since $\varphi^{N}$ is an isogeny, $\psi$ is an isogeny as well. 
\end{pproof}

\begin{prop} \label{prop: general reductive}
Let $f \in K^{\times} \setminus k$ satisfy a nonzero homogeneous linear differential equation $L(f)=0$ over $k$ of order $n$ and assume that the differential Galois group over $k$ of the latter equation is reductive and has a simply connected derived subgroup.  Let $g \in K$ satisfy a nonzero homogeneous linear differential equation over $k$. Assume that $f$ and $g$ are algebraically dependent over $k$. 

\begin{enumerate}
\item If $F$ is a differential field extension of $k(f,g)$ with field of constants $C$ containing a basis of solutions $f_{1}=f,f_{2},\ldots,f_{n}$ of $L$, then 
$$
g \in k[W(f_{1},\ldots,f_{n}),y_{1}^{-1/q},\ldots,y_{s}^{-1/q}]
$$
where $W(f_{1},\ldots,f_{n})$ denotes the wronskian matrix associated to $f_{1},\ldots,f_{n}$, the $y_{i} \in k[W(f_{1},\ldots,f_{n})]$ are such that $y_{i}'/y_{i} \in k$ and $q$ is a positive integer. \ 
\item Assume 
\vspace{.05in}
\begin{center}
(*) $k(f)$ is relatively algebraically closed in  $k(W(f_{1},\ldots,f_{n}),y_{1}^{-1/q},\ldots,y_{s}^{-1/q})$.
\end{center} 
\vspace{.05in}
Then,
$$
g \in k[f].
$$ 

\end{enumerate}
\end{prop}

\begin{remark}
Actually, the proof of Proposition \ref{prop: general reductive} gives the following more precise informations. 
We let $L=L_{1}\cdots L_{r}$ be a factorization  of $L$ over $k$ {into irreducible factors}. The proof of Proposition \ref{prop: general reductive} shows that $s \leq r$ and that the $y_{i}$ can be chosen such that  {$y_{i}$ are solutions of $y'=tr(A_{i}) y$ where $ A_i$ is a certain matrix gauge equivalent to the companion matrix of $L_i(Y) = 0$.}
In particular, with these {refinements} in mind, we recover the main ingredients of the proof of Proposition \ref{prop: case GL} (the only additional ingredient of the proof of Proposition \ref{prop: case GL} is the fact that the hypothesis (*) is satisfied under the assumptions of Proposition \ref{prop: case GL}). 
\end{remark}

\begin{pproof} 
Let $Y'=AY$ be the linear differential system corresponding to $L$. Its differential Galois group over $k$ (the same as $L$) being reductive, we have that $Y'=AY$ is equivalent over $k$ to 
$$
Y'=\overline{A}Y, \ \ \ \overline{A}=\operatorname{diag}(A_{1},\ldots,A_{s}) 
$$ 
for some $A_{i} \in M_{n_{i}}(k)$ such that $Y'=A_{i}Y$ is irreducible over $k$. 

Let $M$ be a nonzero linear differential operator with coefficients in $k$ annihilating $g$ and of minimal order $m$ for this property. Let $Y'=BY$ be the corresponding linear differential system.

We let $E$ be a Picard-Vessiot extension over $k$ of the differential system  
\begin{equation}\label{diff syst direct sum quint}
 Y'=
\begin{pmatrix}
 \overline{A}&0\\
 0&B
\end{pmatrix}
Y
\end{equation}
containing $f$ and $g$. Let $G$ be the corresponding differential Galois group over $k$. The differential system \eqref{diff syst direct sum quint} has a fundamental solution of the form  
$$
\mathfrak{Y} 
=
\begin{pmatrix}
 \mathfrak{Y}_{ \overline{A}} & 0\\
 0 & \mathfrak{Y}_{B} 
\end{pmatrix} 
\in \GL_{m+n}(E)
$$
where $\mathfrak{Y}_{ \overline{A}}$ is a fundamental solution of $Y'=\overline{A} Y$ and  $\mathfrak{Y}_{B}$ is a fundamental solution of $Y'=B Y$ 
(thus, $E$ is the field extension of $k$ generated by the entries of  $\mathfrak{Y}_{\overline{A}}$ and $\mathfrak{Y}_{B}$). 
The field extension $E_{\overline{A}}$ (resp. $E_{B}$) of $k$ generated by the entries of $ \mathfrak{Y}_{ \overline{A}}$ (resp. by the $ \mathfrak{Y}_{B}$) is a  Picard-Vessiot extension over $k$ for $Y'=\overline{A}Y$ (resp. $Y'=BY$); we have $f \in E_{\overline{A}}$ and $g \in E_{B}$. We let $G_{\overline{A}}$ (resp. $G_{B}$) be the corresponding differential Galois group over $k$. We {recall} that $G_{\overline{A}}$ is reductive with simply connected derived subgroup $G_{\overline{A}}'$; we will denote by $d$ the dimension of the torus $Z_{G_{\overline{A}}} ^{0}$. 

By minimality of $M$, we have that $E_{B}$ is generated, as a differential field extension of $k$, by $Gg=G_{B}g$
(we {recall} that $Gg=G_{B}g$ because the restriction morphism $G \rightarrow G_{B}$ is surjective). 
Moreover, $E_{\overline{A}}$ is left invariant by $G$. These facts, together with the fact that $g$ is algebraic over $E_{A}$, imply that $E$ is a finite extension of $E_{\overline{A}}$. 

{We again have} that the restriction morphism  
$
\pi_{\overline{A}}: G \rightarrow G_{\overline{A}}
$
is surjective and hence induces a short exact sequence 
$$
0 \rightarrow J \rightarrow G \xrightarrow[]{\pi_{\overline{A}}} G_{\overline{A}} \rightarrow 0. 
$$
The kernel $J$ (which is the group of differential fields automorphisms of $E$ over $E_{\overline{A}}$) is finite because $E$ is a finite extension of $E_{\overline{A}}$.

Applying Lemma \ref{lem isogeny} to $\pi_{\overline{A}}: G \rightarrow G_{\overline{A}}$, we get that $G$ is reductive and that $Z_{G}^{0}$ is a torus of dimension $d$. 

Let 
$$
\widetilde{A}=\overline{A}- \text{diag}(\operatorname{tr}(A_{1})/n_{1}I_{n_{1}},\ldots,\operatorname{tr}(A_{s})/n_{s}I_{n_{s}})
$$ 
and consider the differential system 
$$
 Y'=
 \text{diag}(\operatorname{tr}(A_{1})/n_{1},\ldots\operatorname{tr}(A_{s})/n_{s},\widetilde{A},B)
Y. 
$$
We let ${H}$ be the differential Galois group over $k$ of this equation. In the tannakian correspondance, 
the equation 
\eqref{diff syst direct sum quint}
 corresponds to the representation $\rho: {H} \rightarrow G$ given by 
$$
\rho:  \text{diag}(c_{1},\ldots,c_{s},\widetilde{g}_{1},\ldots,\widetilde{g}_{s},g) \mapsto  \text{diag}(c_{1}\widetilde{g}_{1},\ldots,c_{s}\widetilde{g}_{s},g). 
$$
Applying Lemma \ref{lem isogeny} to the isogeny $\rho: { H} \rightarrow G$ ({as before,} $\rho$ is onto, and it is clear that it has finite kernel because the $\widetilde{g}_{i}$ have determinant $1$), we have that ${ H}$ is reductive and that $Z_{H}^{0}=\rho^{-1}(Z_{G}^{0})^{0}$ is a torus of dimension $d$. 
We claim that there exists an isogeny $
\Xi:Z_{H}^{0} \rightarrow Z_{G}^{0}
$
of the form 
$$
\text{diag}(c_{1},\ldots,c_{s},\widetilde{g}_{1},\ldots,\widetilde{g}_{s},g) \mapsto
\Xi(c_{1},\ldots,c_{s}), 
$$
an isogeny $\psi: Z_{G}^{0} \rightarrow Z_{G}^{0}$ and a positive integer $N$ such that 
$$
\rho_{\vert Z_{H}^{0}}^{N}=\psi \circ \Xi.
$$
Indeed, the image of $Z_{H}^{0}$ by 
$$
\text{diag}(c_{1},\ldots,c_{s},\widetilde{g}_{1},\ldots,\widetilde{g}_{s},g) \mapsto
\text{diag}(c_{1},\ldots,c_{s})
$$
is a torus of dimension $d$ (this is because $\pi_{\overline{A}} \circ \rho: H \rightarrow G_{\overline{A}}$, which is given by 
$$
\text{diag}(c_{1},\ldots,c_{s},\widetilde{g}_{1},\ldots,\widetilde{g}_{s},g) \mapsto  \text{diag}(c_{1}\widetilde{g}_{1},\ldots,c_{s}\widetilde{g}_{s}) 
$$ 
is an isogeny and, hence, induces an isogeny $Z_{H} \rightarrow Z_{G_{\overline{A}}}$, but $Z_{G_{\overline{A}}}$ is made of diagonal matrices of the form 
$$
\operatorname{diag}(*I_{n_{1}},\ldots,*I_{n_{s}}),
$$ 
and the $\widetilde{g}_{i}$ have determinant $1$, 
so $(\pi_{\overline{A}} \circ \rho_{\vert Z_{H}})^{n_{1}\cdots n_{s}}$ can be described as follows~:
$$
(\pi_{\overline{A}} \circ \rho_{\vert Z_{H}})^{n_{1}\cdots n_{s}}: 
 \text{diag}(c_{1},\ldots,c_{s},\widetilde{g}_{1},\ldots,\widetilde{g}_{s},g) \mapsto  \text{diag}(c_{1}I_{n_{1}},\ldots,c_{s}I_{n_{s}})^{n_{1}\cdots n_{s}}
$$
 and hence the image of $Z_{H}$ by 
$$
\text{diag}(c_{1},\ldots,c_{s},\widetilde{g}_{1},\ldots,\widetilde{g}_{s},g) \mapsto
\text{diag}(c_{1}I_{n_{1}},\ldots,c_{s}I_{n_{s}})
$$
has the same dimension as the image of $\pi_{\overline{A}} \circ \rho_{\vert Z_{H}}$, which is equal to the dimension of $Z_{G_{\overline{A}}}$), 
so there exists 
an isogeny $
\Xi:Z_{H}^{0} \rightarrow Z_{G}^{0}
$
of the form 
$$
\text{diag}(c_{1},\ldots,c_{s},\widetilde{g}_{1},\ldots,\widetilde{g}_{s},g) \mapsto
\Xi(c_{1},\ldots,c_{s}).
$$
But $\rho_{\vert Z_{H}^{0}}: Z_{H}^{0} \rightarrow Z_{G}^{0}$ is also an isogeny, so Lemma \ref{lem : isogenies tori} ensures that there exist a positive integer $N$ and an isogeny $\psi: Z_{G}^{0} \rightarrow Z_{G}^{0}$ such that 
$$
\psi \circ \Xi=\rho_{\vert Z_{H}^{0}}^{N}.
$$
This proves our claim.

Consider the differential system  
$$
 Y'=
 \text{diag}(\operatorname{tr}(A_{1})/(Nn_{1}),\ldots,\operatorname{tr}(A_{s})/(Nn_{s}),\widetilde{A},B) 
Y.
$$
Let $H_{N}$ be the differential Galois group over $k$ of this system. In the tannakian correspondance, 
the equation 
$$
 Y'=
 \text{diag}(A,B)
Y
$$
 corresponds to 
 $$
 \varrho:= \rho \circ p: H_{N} \rightarrow G
 $$ 
 where $p: H_{N} \rightarrow H$ (which corresponds to $Y'=
 \text{diag}(\operatorname{tr}(A_{1})/n_{1},\ldots\operatorname{tr}(A_{s})/n_{s},\widetilde{A},B)
Y$ in the tannakian correspondence) is given by
$$
p: \operatorname{diag}(e_{1},\ldots,e_{s},\widetilde{g}_{1},\ldots,\widetilde{g}_{s},g) \mapsto \operatorname{diag}(e_{1}^{N},\ldots,e_{s}^{N},\widetilde{g}_{1},\ldots,\widetilde{g}_{s},g). 
$$
Applying Lemma \ref{lem isogeny} to $p: H_{N} \rightarrow H$, we get that $H_{N}$ is reductive and that 
$Z_{H_{N}}^{0}=p^{-1}(Z_{H}^{0})^{0}$ is torus of dimension $d$. 

If $\operatorname{diag}(e_{1},\ldots,e_{s},\widetilde{g}_{1},\ldots,\widetilde{g}_{s},g)   \in Z_{H_{N}}^{0}$, then 
\begin{multline*}
 \varrho(\operatorname{diag}(e_{1},\ldots,e_{s},\widetilde{g}_{1},\ldots,\widetilde{g}_{s},g) )^{N}=\rho_{\vert Z_{H}^{0}}(\operatorname{diag}(e_{1}^{N},\ldots,e_{s}^{N},\widetilde{g}_{1},\ldots,\widetilde{g}_{s},g))^{N}\\ =\psi \circ \Xi(e_{1}^{N},\ldots,e_{s}^{N})
=\psi \circ \Xi(e_{1},\ldots,e_{s})^{N}
\end{multline*}
so 
$$
\varrho_{\vert Z_{H_{N}}^{0}}=\psi \circ \Xi.
$$

Moreover, the representation of $H_{N}$ corresponding to $Y'=\overline{A}Y$ in the tannakian correspondence is   
$$
\varpi_{\overline{A}} = \pi_{\overline{A}} \circ \varrho: H_{N} \rightarrow G_{\overline{A}}.
$$
It is surjective (general fact) and has finite kernel (because $\varrho$ and $\pi_{\overline{A}}$ have finite kernels), so it is an isogneny. Therefore, it induces an isogeny 
$$
\varpi_{\overline{A}\vert H_{N}'} = \pi_{\overline{A}} \circ \varrho: H_{N}'\rightarrow G_{\overline{A}}'.
$$
Since $G_{\overline{A}}'$ is a connected and simply connected algebraic group, 
$$
\varpi_{\overline{A}\vert H_{N}'} = \pi_{\overline{A}} \circ \varrho: H_{N}'\rightarrow G_{\overline{A}}'.
$$
is actually an isomorphism. Since $G=\operatorname{im}(\pi_{\overline{A}} \circ \varrho \oplus \pi_{B} \circ \varrho)$, it follows that there exists an algebraic group morphism $u: G_{\overline{A}}' \rightarrow G_{B}'$ such that 
$$
\pi_{B} \circ \varrho= u \circ (\pi_{\overline{A}} \circ \varrho).
$$
So, the restriction of $\varrho$ to the derived subgroup $H_{N}'$ of $H_{N}$ is given by 
$$
\varrho_{\vert H_{N}'}: 
 (1,\widetilde{g}_{1},\ldots,\widetilde{g}_{s},g) \mapsto \operatorname{diag}(\widetilde{g}_{1},\ldots,\widetilde{g}_{s},u(\widetilde{g}_{1},\ldots,\widetilde{g}_{s})).
$$

Finally, we see that the representation $\varpi_{B} =\pi_{B} \circ \varrho$ of $H_{N}$ corresponding to $Y'=BY$ in the tannakian correspondence is   
given by 
$$
\varpi_{B} =\pi_{B} \circ \varrho: \operatorname{diag}(e_{1},\ldots,e_{s},\widetilde{g}_{1},\ldots,\widetilde{g}_{s},g)  \mapsto \psi \circ \Xi(e_{1},\ldots,e_{s}) u(\widetilde{g}_{1},\ldots,\widetilde{g}_{s}).
$$

It follows that the differential system $Y'=BY$ is equivalent over $k$ to 
$$
Y'=\operatorname{diag}(B_{1},\ldots,B_{\ell})Y
$$ 
where each $B_{i} \in M_{m_{i}}(k)$ has the form 
$$
\operatorname{Constr}_{i}(\widetilde{A})+\sum_{j=1}^{s}\alpha_{i,j}\operatorname{tr}(A_{j})/(n_{j}N)I_{m_{j}}
$$ 
for some $\alpha_{i,j} \in \mathbb{Z}$ and some $\operatorname{Constr}_{i}(\widetilde{A})$ obtained from $\widetilde{A}$ by some construction of linear algebra. It follows that the entires of the solutions of $Y'=BY$ belong to $k[\mathfrak{Y}_{\overline{A}},\sqrt[q]{y_{1}}^{-1},\ldots,\sqrt[q]{y_{s}}^{-1}]$ for some positive integer $q$, where each $y_{j}$ satisfies $y_{j}'=\operatorname{tr}(A_{j}) y_{j}$. 
\end{pproof}

\section{Applications}\label{sec:applications}

In this section we will apply the previous results to generalized hypergeometric series and to iterated integrals.
\subsection{Generalized hypergeometric series} 
\label{sec:hyper}

The aim of this Section is to prove Theorem \ref{thm intro : application to alg rel hypergeo} stated in Section \ref{sec : intro}.

\subsubsection{The generalized hypergeometric series and equations}

We {recall} that, for any $p,q \in \ZZ_{\geq 0}$,  the generalized hypergeometric series with parameters $\und \alpha =(\alpha_{1},\ldots,\alpha_{p}) \in \CC^{p}$ and $\und \beta =(\beta_{1},\ldots,\beta_{q}) \in (\CC\setminus \mathbb{Z}_{\leq 0})^{q}$ is given by 
\begin{equation*}
\pFq{p}{q}{\und \alpha}{\und \beta}{x}=\sum_{k=0}^{+\infty} \frac{(\alpha_{1})_k \cdots (\alpha_{p})_k}{(\beta_{1})_k \cdots (\beta_{q})_k} \frac{x^k}{k!} \in \CC((x))
\end{equation*}
where the Pochhammer symbols $(t)_{k}$ are defined by $(t)_0=1$ and, for $k \in \mathbb{Z}_{\geq 1}$, $(t)_k=t(t+1)\cdots (t+k-1)$. 

This series satisfies the generalized hypergeometric differential equation 
$$
\hypergeoequa{p}{q}{\und \alpha}{\und \beta}(\pFq{p}{q}{\und \alpha}{\und \beta}{x}) =0 
$$
where 
\begin{equation*} 
\hypergeoequa{p}{q}{\und \alpha}{\und \beta} = \delta \prod_{k=1}^{q} (\delta + \beta_k-1) - x \prod_{k=1}^{p} (\delta + \alpha_k)
\end{equation*}
with $\delta=x\frac{d}{dx}$.

In that case, $\hypergeoequa{p}{q}{\und \alpha}{\und \beta}$ is a linear differential operator of order $q+1$, and 
\begin{itemize}
\item if $q+1=p$, then $\hypergeoequa{p}{q}{\und \alpha}{\und \beta}$ has order $q+1=p$ and has at most three singularities, namely $0$, $1$ and $\infty$, all regular; 
\item if $q+1>p$, then $\hypergeoequa{p}{q}{\und \alpha}{\und \beta}$ has order $q+1$ and has at most two singularities, namely $0$ and $\infty$; $0$ is regular whereas $\infty$ is irregular. 
\item if $q+1< p$, then $\hypergeoequa{p}{q}{\und \alpha}{\und \beta}$ has order $p$ and has at most two singularities, namely $0$ and $\infty$; $0$ is irregular whereas $\infty$ is regular. 
\end{itemize}

\subsubsection{Formal structure at $\infty$ of the generalized hypergeometric equations when $q+1>p$}\label{sec:formal structure hypergeo}

We shall now {recall} some basic facts concerning the  formal structure at $\infty$ of $\hypergeoequa{p}{q}{\und \alpha}{\und \beta}$  when $q+1>p$. Consider the uniformiser at $\infty$ given by $t=x^{-1}$. We have 
\begin{equation*} 
\hypergeoequa{p}{q}{\und \alpha}{\und \beta} = -t^{-1}\left((-1)^{q}t\delta_{t} \prod_{k=1}^{q} (\delta_{t} - \beta_k+1) -  (-1)^{p}\prod_{k=1}^{p} (\delta_{t} - \alpha_k)\right)
\end{equation*}
with $\delta_{t}=t\frac{d}{dt}$. 
By definition, the Newton polygon $\mathcal{N}_{\infty}(\hypergeoequa{p}{q}{\und \alpha}{\und \beta})$ of $\hypergeoequa{p}{q}{\und \alpha}{\und \beta}$ at $\infty$ is the convex hull of 
\begin{multline*}
 \{(x,y) \in \RR^{2} \ \vert \ \text{there is a monomial } t^{m} \delta_{t}^{n} \\ 
 \text{ in } \hypergeoequa{p}{q}{\und \alpha}{\und \beta} \text{ with } (x,y) \geq (n,m) \}
\end{multline*}
where $(x_{1},y_{1}) \geq (x_{2},y_{2})$ if and only if $x_{1} \leq x_{2}$ and $y_{1} \geq y_{2}$. 
This polygon has three extremal points, namely $(0,0)$, $(p,0)$ and $(q+1,1)$. Thus, the slopes of $\hypergeoequa{p}{q}{\und \alpha}{\und \beta}$ at $\infty$ are 
$$
\lambda_{1}=0 \text{ with multiplicity $p$ and } \lambda_{2}=1/\sigma \text{ with multiplicity $\sigma$,}
$$  
where
$$
\sigma=q-p+1. 
$$ 
It follows that the differential module $\mathcal{M}$ over $\CC(x)$ associated to $\hypergeoequa{p}{q}{\und \alpha}{\und \beta}$ satisfies 
$$
\CC((t)) \otimes_{\CC(x)} \mathcal{M} = \widehat{\mathcal{M}}_{\lambda_{1}} \oplus \widehat{\mathcal{M}}_{\lambda_{2}}
$$
where $\widehat{\mathcal{M}}_{\lambda_{1}}$ is regular singular of rank $p$ and $\widehat{\mathcal{M}}_{\lambda_{2}}$ is irregular of rank $\sigma$, with only one slope, namely $\lambda_{2}$. According to \cite[Remark 3.34]{VdPS}, we have 
$$
\CC((t^{1/\sigma})) \otimes_{\CC((x))} \widehat{\mathcal{M}}_{\lambda_{2}}=\oplus_{j=1}^{\sigma} \widehat{\mathcal{R}} \otimes_{\CC((t^{1/\sigma}))} \widehat{\mathcal{E}}(q_{j}(t^{-1/\sigma}))
$$ 
where: 
\begin{itemize}
 \item $\widehat{\mathcal{R}}$ is a regular singular differential module over $\CC((t^{1/\sigma}))$ of rank one, 
 \item for all $j \in \{1,\ldots,\sigma\}$, $q_{j}(t^{-1/\sigma})=q(\zeta_{j} t^{-1/\sigma})$ with $\zeta_{j} = e^{\frac{2 \pi i j}{\sigma}}$ for some $q(X) \in X\CC[X]$ of $X$-adic valuation $1$, 
\item $\widehat{\mathcal{E}}(q_{j}(t^{-1/\sigma}))$ is the rank one differential module over $\CC((t^{1/\sigma}))$ defined by $\widehat{\mathcal{E}}(q_{j}(t^{-1/\sigma}))=\CC((t^{1/\sigma})) e_{j}$ with $\partial (e_{j})=q_{j}(t^{-1/\sigma})e_{j}$. 
\end{itemize}
{Recall} that the slopes of $\widehat{\mathcal{M}}_{\lambda_{2}}$ can be computed using the $q_{j}(t^{-1/\sigma})$: they are the {negative} of the $t$-adic valuations of the $q_{j}(t^{-1/\sigma})$, which are equal all to $\frac{\deg_{X} q(X)}{\sigma}$. Since $\widehat{\mathcal{M}}_{\lambda_{2}}$ has a unique slope $\lambda_{2}=1/\sigma$, we get $\deg_{X} q(X)=1$ so $q(X)=aX$ is a monomial and  
$$
q_{j}(t^{-1/\sigma})=a \zeta_{j} t^{-1/\sigma}. 
$$

We shall now determine $a$. 
Set $w=t^{1/\sigma}$ and $\delta_{w}=w\frac{d}{dw}=\sigma \delta_{t}$. We have to find the $a$ such that the Newton polygon (with respect to $w$) of $e^{-aw^{-1}}\hypergeoequa{p}{q}{\und \alpha}{\und \beta} e^{aw^{-1}}$ has height $<q+1$ (because the slopes of the latter differential operator are $0$ and/or $1$). 
We have  
\begin{multline*}
 \hypergeoequa{p}{q}{\und \alpha}{\und \beta} =\\ -w^{-\sigma}\left((-1)^{q}w^{\sigma}\frac{\delta_{w}}{\sigma} \prod_{k=1}^{q} \left(\frac{\delta_{w}}{\sigma}  - \beta_k+1\right) -  (-1)^{p}\prod_{k=1}^{p} \left(\frac{\delta_{w}}{\sigma}  - \alpha_k\right)\right)
\end{multline*}
and 
\begin{multline*}
e^{-aw^{-1}}\hypergeoequa{p}{q}{\und \alpha}{\und \beta} e^{aw^{-1}} =\\ 
-w^{-\sigma}\left((-1)^{q}w^{\sigma}\left(\frac{\delta_{w}}{\sigma}-\frac{a}{\sigma}w^{-1}\right) \prod_{k=1}^{q} \left(\frac{\delta_{w}}{\sigma} -\frac{a}{\sigma}w^{-1} - \beta_k+1\right)\right. \\ 
\left. -  (-1)^{p}\prod_{k=1}^{p} \left(\frac{\delta_{w}}{\sigma}  -\frac{a}{\sigma}w^{-1} - \alpha_k\right)\right)
\end{multline*}
The coefficient of $\delta_{w}^{j}$ in $e^{-aw^{-1}}\hypergeoequa{p}{q}{\und \alpha}{\und \beta} e^{aw^{-1}}$ has $w$-valuation $\geq  -q+j-1$ and the $w$-valuation of $\delta_{w}^{q+1}$ is equal to $0$;  therefore, we have to find the $a$ such that the coefficient of $\delta_{w}^{0}$ in $e^{-aw^{-1}}\hypergeoequa{p}{q}{\und \alpha}{\und \beta} e^{aw^{-1}}$ has valuation $>q+1$. But the latter coefficient is of the form 
$$
\left(\frac{a}{\sigma}w^{-1}\right)^{q+1} - w^{-\sigma} \left(\frac{a}{\sigma}w^{-1}\right)^{p} + \text{ terms of higher degree in $w$}; 
$$ 
it has valuation $>q+1$ if and only if $(\frac{a}{\sigma})^{q+1}=(\frac{a}{\sigma})^{p}$ if and only if $a \in \sigma \mu_{\sigma}$. 

In conclusion, the list of determining polynomials\footnote{These ``determining polynomials'' are the ``eigenvalues'' of \cite[Definition 3.26]{VdPS}.} at $\infty$ of $\hypergeoequa{p}{q}{\und \alpha}{\und \beta}$ is 
$$
(0 \text{ repeated $p$ times},  \zeta_{1} \sigma t^{-1/\sigma},\ldots, \zeta_{\sigma} \sigma t^{-1/\sigma}) 
$$
with $\zeta_{j} = e^{\frac{2 \pi i j}{\sigma}}$.

\subsubsection{A preliminary remark concerning Proposition \ref{prop sln} in the $\SL_{2}(C)$ case}\label{subsec:rem on SL2 case}

The hypothesis of Proposition \ref{prop sln} is:
\begin{quote}
Let $f \in K^{\times}$ satisfy a nonzero homogeneous linear differential equation $L(f)=0$ over $k$ of order $n\geq 1$ and assume that the differential Galois group $G_{L}$ over $k$ of the latter equation is simply connected. Let $g \in K$ satisfy a nonzero homogeneous linear differential equation over $k$. Assume that $f \not \in k$ and that $f$ and $g$ are algebraically dependent over $k$. 
\end{quote}
With the notations of the proof of Proposition \ref{prop sln}, we have 
$$
\pi_{M}=u \circ \pi_{L}. 
$$ 
We shall now focus on the case $G_{L}=\SL_{2}(\CC)$. Then, the classification of the representations of $\SL_{2}(\CC)$ shows that the representation $u: G_{L} \rightarrow G_{M}$ is conjugate to the direct sum of symmetric power representations 
$$
\operatorname{Sym}^{m_{i}} : G_{L}=\SL_{2}(\CC) \rightarrow \SL_{m_{i}+1}(\CC).
$$ 
{Letting $\cM$ and $\cL$ denote the differential modules associated with $Y'= A_MY$ and $Y' = A_LY$, the tannakian correspondence implies that $\cM$ is isomorphic to the direct sum of the $\operatorname{Sym}^{m_{i}}(\cL)$}. We will now use this to study the algebraic relations between generalized hypergeometric series. 

\subsubsection{Proof of Theorem \ref{thm intro : application to alg rel hypergeo}}

The differential Galois group over $\overline{\CC(x)}$ of $L=\hypergeoequa{0}{1}{-}{\beta}$ is $\SL_{2}(\CC)$; see \cite[Theorem 3.6]{expsum}. As explained in Section \ref{subsec:rem on SL2 case}, the {differential module $\mathcal{M}$ associated to the} minimal nonzero differential equation $M$ with coefficients in $\overline{\CC(x)}$  annihilating $\pFq{p}{q}{\und \gamma}{\und \delta}{x}$ is isomorphic to a direct sum of symmetric powers of $\mathcal{L}$, say  
$$
\mathcal{M} \cong \oplus_{i=1}^{r} \operatorname{Sym}^{m_{i}}(\mathcal{L}).
$$

Let us first assume that $q+1 > p$. 
We have seen in Section \ref{sec:formal structure hypergeo} that $L$ is irregular at $\infty$: it has exactly one slope at $\infty$, namely $1/2$, and its list of determining polynomials at $\infty$ is $\pm 2 x^{1/2}$. Therefore, the list of the determining polynomials of $\operatorname{Sym}^{m_{i}}(\mathcal{L})$ at $\infty$  is 
\begin{equation}\label{det fact symm}
  -2m_{i}x^{1/2},2(-m_{i}+2)x^{1/2},2(-m_{i}+4)x^{1/2},\ldots,2(m_{i}-2)x^{1/2},2m_{i}x^{1/2}. 
\end{equation}
So, the list of determining polynomials of $\mathcal{M}$ is the concatenation of the lists \eqref{det fact symm} for $i$ varying in $\{1,\ldots,r\}$. 

On the other hand, $M$ is a factor of $\hypergeoequa{p}{q}{\und \gamma}{\und \delta}$ so (see Section \ref{subsec:rem on SL2 case}) the list of determining polynomials of $\mathcal{M}$ is a sublist of 
 $$
 0 \text{ with multplicity } p, \zeta_{1} \sigma x^{1/\sigma}, \ldots,\zeta_{\sigma} \sigma x^{1/\sigma}
 $$
with $\sigma=q-p+1$ and $\zeta_{j}=e^{\frac{2 \pi i j}{\sigma}}$. 

Comparing the two preceding descriptions of the determining factors of $\mathcal{M}$, we find that $\sigma=2$ and that we have 
\begin{itemize}
 \item either all the $m_{i}$ are equal to $0$;
  \item or one of  the $m_{i}$, say $m_{r}$, is equal to $1$ and the other $m_{i}$ are equal to $0$. 
\end{itemize}
 We claim that the first case cannot happen. 
Indeed, otherwise the entire function $\pFq{p}{q}{\und \gamma}{\und \delta}{x}$ would be algebraic and, hence, polynomial, which is false. 

Therefore, we have  
$$
\cM \cong \operatorname{Sym}^{0}(\cL) \oplus \cdots \oplus \operatorname{Sym}^{0}(\cL) \oplus \operatorname{Sym}^{1}(\cL). 
$$
It follows that $\pFq{p}{q}{\und \gamma}{\und \delta}{x}$ is of the form $a+bg+cg'$  for some $a,b,c \in \overline{\CC(x)}$, $b$ or $c \neq 0$, and some nonzero solution $g$ of $L$. Since the differential Galois group of $L$ over $\overline{\CC(x)}$ is $\SL_{2}(\CC)$, the only possibility for $a+bg+cg'$ to be algebraically depend with $\pFq{0}{1}{\alpha}{\beta}{x}$ is that $g= d \cdot \pFq{0}{1}{-}{\beta}{x}$ for some $d \in \CC$ and $c=0$. Therefore, 
$$
\pFq{p}{q}{\und \gamma}{\und \delta}{x}  \in \overline{\CC(x)} \pFq{0}{1}{-}{\beta}{x} 
+
\overline{\CC(x)}.
$$
The fact that we can descend this linear relation to $\CC(x)$ follows from the fact that $\pFq{p}{q}{\und \gamma}{\und \delta}{x}$ and $\pFq{0}{1}{-}{\beta}{x}$ are entire functions and that $\pFq{0}{1}{-}{\beta}{x}$ and $1$ are linearly independent over $\overline{\CC(x)}$. 

Let us now consider the case $q+1 \leq p$. In that case, we have seen that $\hypergeoequa{p}{q}{\und \gamma}{\und \delta}$ is regular at $\infty$. Since $M$ is a factor of $\hypergeoequa{p}{q}{\und \gamma}{\und \delta}$, $M$ and hence $\mathcal{M}$ are regular at $\infty$ as well. Since $\operatorname{Sym}^{m_{i}}(\mathcal{L})$ is irregular at $\infty$ if $m_{i}$ is nonzero, we infer that all the $m_{i}$ are equal to $0$. Therefore, $\pFq{p}{q}{\und \gamma}{\und \delta}{x}$ belongs to $\overline{\CC(x)}$. This excludes the case $q+1<p$, because the radius of convergence of $\pFq{p}{q}{\und \gamma}{\und \delta}{x}$ is $0$ in that case. 

\subsection{Iterated integrals}\label{sec:multint} From \cite[Chapter XVIII, Theorem 3.3]{Hochschild} we know that a unipotent group is simply connected.  This allows us to make an  application of Proposition~\ref{prop case 2} when the Galois group is unipotent.   %
 \begin{defin}[Cf.~\cite{Ravi}]   We say that $f\in K$ is an {\rm iterated integral} over $k$ if for, for some $n \in \NX$, $f^{(n)} \in k$. \end{defin}

\begin{thm} \label{prop:itint}Let $f \in K^\times$ be a iterated integral over $k$ and let $g \in K$ satisfy a nonzero homogeneous linear differential equation over $k$.  Assume that $k$ contains an element $x$ with $x'=1$, $f$ is transcendental over $k$ and that $f$ and $g$ are algebraically dependent over $k$.  Then $g \in k[f]$. \end{thm} 
 \begin{remark} The condition that $k$ contains an element $x$ such that $x'=1$ is needed. For example, let $f = \CX$ and $K = \CX(x), x' = 1$.  The element  $x^2$ is an iterated integral over $k$, $x$ satisfies a linear differential equation over $k$  and these elements are algebraically dependent over $k$.  Nonetheless, $x \notin \CX[x^2]$.\end{remark}
 We will use, again and again, the fact that if $E\subset F$ are differential fields with the same algebraically closed field of  constants and $y \in F$ with $y'\in E$, then either $y \in E$ or $y$ is transcendental over $E$.  This follows easily from the Galois theory and the fact that the only algebraic subgroups  of $\Ga(C)$ are $\{0\}$ and $\CX$.
 \begin{lem}\label{lem:simple_int} Let $E \subset F$ be differential fields with the same algebraically closed subfield $C$ of  constants. If $x \in k$ such that $x'=1$ and  $y,z \in F$ such that
 \begin{enumerate}
 \item $y'\in E$ and $y$ transcendental over $E$ and
 \item $z \in E(y)$ and $z'=y,$ 
 \end{enumerate}
 then $z= Ay+B$ for some $A, B \in E$.
 \end{lem}
 
 \begin{proof} Let $z= p(y)/q(y)$ where $p$ and $q$ are relatively prime polynomials. As in the proof of Corollary~\ref{cor f dans k de g},  \cite{HarrisSibuya85} implies that $q(y)'/q(y)$ is algebraic over $E$ and therefore in $E$ (since $E$ is algebraically closed in $E(y))$. We then have that $E(q(y))$ is a Picard-Vessiot extension of $E$ whose differential Galois group $G$ is a subgroup of $\GL_1(C)$ and is a quotient of $\Ga(C)$. This implies that $G$ must  be trivial so $q(y) \in E$.
 
 Therefore we may write $z = a_my^m + \ldots + a_0$ with the $a_i \in E$. Differentiating, we have
 \[y = z' = a_m'y^m + (ma_m y' + a_{m-1}')y^{m-1} + \ldots .\]
 If $a_m' \neq 0$ we have $m = 1$ and our conclusion follows. If $a_m' = 0$  we    must have $m \geq 2$ and we will derive a contradiction. We have that 
 \[ ma_m y' + a_{m-1}' = c \in E\]
 where $c = 0$ if $m>2$ or $c=1$ if $m=2$. In either case, $y = \frac{1}{ma_m}(cx -a_{m-1} +d)\in E$ for some constant $d$, contradicting the fact that $y$ is transcendental over $E$. \end{proof}

 \begin{lem}\label{lem:algindep} Let $k\subset K$ be differential fields with the same algebraically closed subfield of constants $C$. Assume   that there exists $x \in k$ such that $x' = 1$. \begin{enumerate}
 \item[1.]\label{lem:mult1} If $f$ is a iterated integral over $k$, then $k\langle f\rangle$ is a Picard-Vessiot extension of $k$ with unipotent Galois group. 
 \item[2.] If $f \notin k$, then there exist algebraically independent $y_1, \ldots y_{r-1}, y_r = f$  such that $  k\langle f\rangle = k(y_1, \ldots , y_r)$. In particular, $k(f)$ is algebraicaly closed in $ k\langle f\rangle$.
 \end{enumerate} 
 \end{lem}
 \begin{proof} If $f^{(n)} = h \in k$ then $f$ satisfies 
\begin{align} \label{eq:int}
y^{(n+1)} - \frac{h'}{h} y^{(n)} &= 0.
\end{align}
(1) If $f \in k$, then clearly $k\langle f\rangle$ is a Picard-Vessiot extension of $k$. If $f \notin k$ then $f, 1, x, \ldots x^{n-1}$ are linearly independent over $C$ and so form a basis of the solution space of  \eqref{eq:int} and $k\langle f\rangle$ is a Picard-Vessiot extension of $k$. In addition, if we define $k_i = k_{i-1}(f^{(n-i)})$, we have a tower of differential fields $k = k_0 \subset k_1\subset  \ldots \subset k_n = k\langle f\rangle$ where each $k_i$  is a Picard -Vessiot extension whose differential Galois group is either $\{0\}$ or $\Ga$.  Therefore the differential Galois group of $k\langle f\rangle$ over $k$ is unipotent.
 
 (2) Assume that the transcendence degree of $k\langle f\rangle$ over $k$ is $r$.  By assumption $r \geq 1$. Using the construction of the $k_i$ we see that from  the set $\{ f^{(n-1)},f^{(n-2)}, \ldots , f', f \}$ we may select a subset of elements $ T= \{y_i =  f^{(n-n_i)}\ | \ n_1<n_2< \ldots < n_r\}$ that forms a transcendence basis. We wish to show that we can select these so that $n_r = n$. Suppose among all choices of the set $T$ we have selected one with $n_r$ maximal.  If $n_r < n$, let $E = k(y_1, \ldots, y_{r-1}), y = y_r = f^{(n-n_r)}, z = f^{(n-(n_r+1))}$. Applying  Lemma~\ref{lem:simple_int}, we have that $y_r = az+b$ for some $a,b \in E$.  Therefore 

\begin{align*} 
k\langle f\rangle &= k(f^{(n-n_1)}, \ldots, f^{(n-n_{r-1})}, f^{(n-n_r)})\\
&=E(y_r) \\
&= E(z) \\
&= k(f^{(n-n_1)}, \ldots,f^{(n-n_{r-1})}, f^{(n-(n_r+1))}),
\end{align*}
contradicting the maximality of $n_r$. 
\end{proof}

\vspace{.2in}

\begin{proofitint} This now follows from Proposition~\ref{prop sln}\end{proofitint}

{\bf For the rest of Section~\ref{sec:multint} we remove the assumption that $k$ is algebraically closed and only assume it is a differential field of charactersitic zero with an algebraically closed subfield of constants.}

A consequence of the Kolchin-Ostrowski Theorem (\cite[Section 2]{Kolchin68}) is that if $k \subset K$ are differential fields with the same constants and $f_1, \ldots , f_n \in K$ with $f_i' \in k$ for $i = 1, \ldots n$ then the $f_i$ are algebraicaly dependent over $k$ if and only if there exist constants $c_i$, not all zero,  such that $c_1f_1 + \ldots + c_nf_n \in k$. Proposition~\ref{prop:liouv}  below gives  a related  result for iterated integrals.

\begin{cor}\label{cor:lindep} Let $k \subset K$ have the same field of constants and assume that there is an element $x\in k$ such that $x'=1$.  Let $f, g \in K$ be iterated integrals over $k$.  If $f$ and $g$ are algebraically dependent over $k$ then there exist $u,v \in k$, not both zero, such that $uf+vg \in k$.\end{cor}
\begin{pproof} Let $E = k(f, f', \ldots , f^{(n-1)}, g, g', \ldots  , g^{(m-1)})$.  Proceeding in a manner similar to the proof of  Lemma~\ref{lem:algindep}.2, there is an algeraically independent set of elements $\{z_1, \ldots , z_t\} \subset \{f, f', \ldots , f^{(n-1)}, g, g', \ldots  , g^{(m-1)}\}$ such that $E = k(z_1, \ldots , z_t)$, that is, $E$ is a purely transcendental extension of $k$. Therefore if either $f$ or $g$ are algebraic, then both of them will lie in $k$ and the conclusion follows. Therefore we may assume neither $f$ nor $g$ are algebriac over $k$.

 Theorem~\ref{prop:itint} implies that $f = P(g)$ and $g = Q(f)$ for some polynomials $P, Q \in \bar{k}[Y]$ where $\bar{k}$ is the algebraic closure of $k$. This implies that $P$ and $Q$ are linear polynomials so $f$, $g$ and $1$ are linearly dependent over $\bar{k}$. Since $E$ is a purely transcendental extension of $k$, it is linearly disjoint from $\bar{k}$ over $k$.  Therefore $f$, $g$ and $1$ are linearly dependent over $k$. \end{pproof}
 
 \begin{example} 
 Let $k = \CX(x), K = \CX(x, \log x)$ and let $f = \log x +x, g = x\log x$. We have $f', g'' \in k$ and $xf-g = x^2 \in k$.
  \end{example}

\begin{prop}\label{prop:liouv} Let $k \subset K$ have the same field of constants and assume that there is an element $x\in k$ such that $x'=1$. Let $f_1, \ldots , f_n$ be iterated integrals over $k$. If $f_1, \ldots , f_n$  are algebraically dependent over $k$ then there exist $u_1, \ldots , u_n \in k$, not all zero, such that $u_1f_1+ \ldots + u_n f_n \in k$. \end{prop}%

\begin{pproof} We can assume that for each pair $i, j, \ i\neq j$ 
\begin{itemize}
\item the set $\{f_1, \ldots , \widehat{f_i}, \ldots , \widehat{f_j}, \ldots , f_n\}$ is algebraically independent over $k$, 
\item $f_i, f_j$ are transcendental over $F_{i,j} = k(f_1, \ldots , \widehat{f_i}, \ldots , \widehat{f_j}, \ldots , f_n)$, and
\item $f_i$ and $f_j$ are algebraically dependent over $F_{i,j}$.
\end{itemize}
Corollary~\ref{cor:lindep} implies that 
\begin{align*}
f_i =& A_{i,j}f_j + B_{i,j} \text{ for some } A_{i,j},B _{i,j} \in F_{i,j}.
\end{align*}
In particular, $f_1 = A _{1,2}f_2 + B_{1,2}\text{ for some } A_{1,2},B _{1,2} \in F_{1,2}$.  We shall show that $A_{1,2} \in k$ and that $B_{1,2}$ is a $k$ linear combination of $f_3, \ldots, f_n, 1$. Note that $k(f_2, \ldots, f_n)$ is a purely transcendental extension of $k$.  We will let $\partial_i$ denote the derivation $\frac{\partial}{\partial f_i}$ on this field.

\underline{$A_{1,2} \in k$.} Since we also have $f_1 = A_{1,3}f_3 + B_{1,3} \text{ for some } A_{1,2},B _{1,2} \in F_{1,3}$, $A_{1,3}f_3 + B_{1,3} = A _{1,2}f_2 + B_{1,2}$. Applying $\partial_3$, we have $A_{1,3} = \p3(A_{1,2}) f_2 + \p3(B_{1,2})$. Since $A_{1,3}$ is independent of $f_3$ and $A _{1,2}$ is independent of $f_1, f_2$ we have $A_{1,2} \in k(y_4, \ldots , y_n)$ so $A_{1,2} = \alpha f_3 + \beta$, where $\alpha, \beta \in k(f_4, \ldots , f_n)$. 

We also have $f_2 =A _{2,1}f_1 + B_{2,1}\text{ for some } A_{2,1},B _{2,1} \in F_{1,2} = F_{2,1}$ and  $f_2 =A _{2,3}f_3 + B_{2,3}\text{ for some } A_{2,3},B _{2,3} \in F_{2,3}$. Comparing these two expressions in a manner similar to the above we have that $A_{2,1} = \bar{\alpha} f_3 + \bar{\beta}, \alpha, \beta \in k(f_4, \ldots , f_n)$.  Since $A_{2,1} = A_{1,2}^{-1}$, we must have $\alpha = \bar{\alpha} = 0$, that is, $A_{1,2}$ is independent of $f_3$.  Replacing $f_3$ with each $f_i, i =4, \ldots, n$ in the above argument yields that $A_{1,2} \in k$.

\underline{$B_{1,2}$ is a $k$-linear combination of $f_3, \ldots, f_n, 1$.} Note that the above argument can be modified to work equaly well for $f_1, f_j, \ j>1$ to conclude that $f_1 = A_{1,j} f_j + B_{1,j}$ where $A_{1,j} \in k$ and $B_{1,j} \in F_{1,j}$.  Therefore $\partial_j(f_1)  = A_{1,j} \in k$, that is, $f_1 \in  k(f_2, \ldots, f_n)$ is linear in each of the $f_i, 2\leq i \leq n$ and so the conclusion is obtained. \end{pproof}

\appendix
\section{The tannakian correspondence}\label{tannaka} We give an informal description of this correspondence, enough to understand its use in our proofs. For a formal description see \cite{delignemilne}, \cite{KatzAlgSolDiffEqua,katzcalculation}, \cite[Ch.~2 and Appendix B]{VdPS}. \\

Let $k$ be a differential field of characterstic zero (not assumed to be algebraically closed)  with algebraically closed field of constants $C$. Let 
\begin{align}\label{eq:app1}
Y' =& AY
\end{align}
with $A \in \gl_n(k)$ and let $K$ be its Picard-Vessiot extension. If $\mathfrak{Y} \in \GL_n(K)$ is a funadamental solution matrix of \eqref{eq:app1}, we refer to the ring $k[\mathfrak{Y}, \det(\mathfrak{Y})^{-1}]$ generated by the entries of $\mathfrak{Y}$ and the inverse of its determinant as the Picard-Vessiot ring associated with \eqref{eq:app1}. It is independent of our choice of fundamental solution matrix.

One can associate to  \eqref{eq:app1}  a differential module $M$, that is, a finite dimensional $k$-vector space $M$ endowed with a map $\partial : M\rightarrow M$ such that for all $f \in k$ and $m,n \in M$, $\partial (m+n) = \partial m + \partial n$ and $\partial(fm) = f'm + f \partial m$. This is done in the following way. Let $M = k^n$ and denote the standard basis by $\{e_i\}$. We define $\partial$ by setting $\partial(e_i):= -\sum_j a_{j,i} e_j$ where $A = (a_{i,j})$. Given any differential module together with a basis $\mathbf{e} =\{e_i\}$, one can reverse this process and produce a differential equation $Y'= A_{\mathbf e} Y$. If ${\mathbf f} = B{\mathbf e}$ is another basis of $M$ with $B \in \GL_n(k)$ then $A_{\mathbf e}$ and $A_{\mathbf f}$ are related by $A_{\mathbf e} = B^{-1}A_{\mathbf f}B- B^{-1}B$. In this case we say that the equations are {\it gauge equivalent} or just {\it equivalent}. \\

Given differential modules $M_1,M_2$ one can define a differential module homomorphism as well as sub and quotient differential modules in the obvious way. One can also form    direct sums, tensor products, and duals as in the following table:
\def\arraystretch{1.5}
$$
\begin{tabu}{ c | c | c }
  \hline
 \text{Construction} & \partial & \text{Equation}  \\
 \hline
 M_1 \oplus M_2 &\partial(m_1\oplus m_2) = \partial_1m_1 \oplus \partial_2m_2 & Y' = \begin{pmatrix} A_1 & 0 \\ 0 & A_2\end{pmatrix}Y\\
  \hline
  M_1 \otimes M_2 & \partial(m_1\otimes m_2 ) =& Y' = \\
    & \partial_1m_1 \otimes m_2 + m_1 \otimes \partial_2m_2 & (A_1 \otimes I_{m_2} + I_{m_1}\otimes A_2)Y\\
    \hline
   M_1^* & \partial(f)(m) = f(\partial_1(m)) & Y' = -A^TY\\
   \hline
\end{tabu}
$$

The collection of all differential modules forms a category closed under differential module homomorphism and the constructions of  submodules, quotients, direct sums, tensor products, and duals, which we refer to as {\it the constructions of linear algebra}. We denote by $\{\{M\}\}$ the smallest subcategory containing $M$ and closed under these five constructions. Given any $N \in \{\{M\}\}$ and any basis ${\mathbf e}$ of $N$, we can form the associated differential equation $Y' = A_{N,{\mathbf e}} Y$, $A_{N,{\mathbf e}}\in \gl_m(k)$. It is known (\cite[Ch.~2.4]{VdPS}) that this equation has a fundamental matrix $\mathfrak{Y}_{N,\mathbf{e}} \in \GL_m(K)$ and, {\it a fortiori}, in the Picard-Vessiot ring $k[\mathfrak{Y}, \det(\mathfrak{Y})^{-1}]$ (\cite[Corollary 1.38]{VdPS}).  The differential Galois group $G$ of $K$ over $k$ acts on the $C$-space $V$ spanned by the columns of $\mathfrak{Y}$ and so induces a representation of $G$. The space $V$ depends on our selection of bases of $N$ but it can be shown that the representation of $G$ is independent, up to $G$-isomorphism, of this choice of bases (see \cite[Ch.~2.4]{VdPS} for a basis free way of defining this representation). We will denote this representation of $G$ by $S(N)$.

We denote by $\rep_G$ the category of representations of $G$, that is the category of finite dimensional $C$-vector spaces on which $G$ acts. In this category one also has a notion of the five constructions above.  The tannakian correspondence \cite[Theorem 2.22]{VdPS} says that
\begin{quotation} The map $N \mapsto S(N)$ is a bijective map between  elements of $\{\{M\}\}$ {(modulo isomorphism of differential  modules)} and  representations of $G$ {(modulo conjugacy of representations)}. This map is compatible with homomorphisms and the five constructions above.
\end{quotation} 
One consequence of this is: if $k$ is algebraically closed, then the Picard-Vessiot ring is naturally isomorphic to the coordinate ring (over $k$)  of the differential Galois group.

\begin{example}\label{ex:appendix} Consider  $Y' = AY$ with $A \in \gl_n(k)$ and let $\tilde{A} = A -(\tr(A)/n)I_n$. Since $\tr(\tilde{A}) = 0$, the differential Galois group of the equation $Y' = \tilde{A}Y$ is unimodular (\cite[Exercise 1.35.5]{VdPS}) and this construction is often used to reduce questions about general linear differential equations to ones that have unimodular Galois groups. 
We now form the $2n\times 2n$ differential system
\begin{align}
Y' =& \mathrm{diag}((\tr(A)/n)I_n, \tilde{A}) Y
\end{align}
and let $\tilde{K}$ be its Picard-Vessiot extension with differential Galois group $H$. Let $K$ be the Picard-Vessiot  extension for $Y'=AY$ and let $G$ be its differential Galois group.

If $\mathfrak{Y}$ is the fundamental solution matrix of $Y'=AY$, then $\mathfrak{y} = \sqrt[n]{\det \mathfrak{Y}}$ is a solution of {$y'= (\tr{A}/n)y$}. Note that $\mathfrak{y}$ need not lie in $K$ but $\mathfrak{y}^n \in K$. A computation shows that if $\tilde{\mathfrak{Y}}$ is a fundamental solution matrix of $y' = \tilde{A}y$, then $\mathfrak{y}I_n\cdot \tilde{\mathfrak{Y}}$ is a fundamental solution matrix of $Y' = AY$.  Therefore $K \subset \tilde{K}$. 
The differential Galois theory implies that restricting $H$ to $K$ gives a short exact sequence 
\begin{align*}
0\rightarrow H_0 \rightarrow H \xrightarrow[]{\rho} G \rightarrow 0
\end{align*}
where $H_0$ is a finite  subgroup of $\GL_n(C)$  (since $\tilde{K}$ is a simple radical extension of $K$) and the map $\rho$ is given by 
\begin{align*}
\rho: \mathrm{diag} (t,h) \rightarrow th.
\end{align*}
This defines a homomorphism  of $H \subset \GL_{2n}(C)$ to $\GL_n(C)$ and so defines a representation of $H$.

This example appears in the proof of Proposition~\ref{prop: case GL}.
\end{example}

\section*{Acknowledgement} The work of the first author was supported by the ANR De rerum natura project, grant ANR-19-CE40-0018 of the French Agence Nationale de la Recherche. The second author was partially supported by a grant from the Simons Foundation (\#349357, Michael Singer).

\bibliography{biblio}

 \end{document}